\documentclass[a4paper,11pt]{article}
\usepackage{amsmath, amsfonts, amscd, amssymb, amsthm, enumerate, xypic}

\def\bfB{\mathbf{B}}
\def\bfC{\mathbf{C}}
\def\loc{\text{loc}}
\def\rc{\text{rc}}

\newcommand{\Mat}{\operatorname{M}}
\newcommand{\Mats}{\operatorname{S}}
\newcommand{\Mata}{\operatorname{A}}

\newcommand{\id}{\operatorname{id}}
\newcommand{\GL}{\operatorname{GL}}
\newcommand{\Ker}{\operatorname{Ker}}
\newcommand{\Vect}{\operatorname{span}}
\newcommand{\im}{\operatorname{Im}}
\newcommand{\urk}{\operatorname{urk}}
\newcommand{\lrk}{\operatorname{lrk}}
\newcommand{\tr}{\operatorname{tr}}

\newcommand{\rk}{\operatorname{rk}}
\newcommand{\codim}{\operatorname{codim}}
\renewcommand{\setminus}{\smallsetminus}
\newcommand{\modu}{\operatorname{mod}}

\def\F{\mathbb{F}}
\def\K{\mathbb{K}}

\def\calA{\mathcal{A}}
\def\calB{\mathcal{B}}

\def\calD{\mathcal{D}}
\def\calE{\mathcal{E}}
\def\calF{\mathcal{F}}
\def\calG{\mathcal{G}}
\def\calH{\mathcal{H}}
\def\calI{\mathcal{I}}
\def\calJ{\mathcal{J}}

\def\calL{\mathcal{L}}
\def\calM{\mathcal{M}}
\def\calN{\mathcal{N}}

\def\calP{\mathcal{P}}

\def\calR{\mathcal{R}}
\def\calS{\mathcal{S}}
\def\calT{\mathcal{T}}
\def\calU{\mathcal{U}}
\def\calV{\mathcal{V}}
\def\calW{\mathcal{W}}
\def\calX{\mathcal{X}}

\def\lcro{\mathopen{[\![}}
\def\rcro{\mathclose{]\!]}}

\theoremstyle{definition}
\newtheorem{Def}{Definition}
\newtheorem{Not}[Def]{Notation}

\theoremstyle{plain}
\newtheorem{theo}{Theorem}[section]
\newtheorem{prop}[theo]{Proposition}

\newtheorem{lemma}[theo]{Lemma}
\newtheorem{claim}{Claim}

\theoremstyle{plain}

\theoremstyle{remark}
\newtheorem{Rems}{Remarks}
\newtheorem{Rem}[Rems]{Remark}

\title{Range-compatible homomorphisms over fields with two elements}
\author{Cl\'ement de Seguins Pazzis\footnote{Universit\'e de Versailles Saint-Quentin-en-Yvelines, Laboratoire de Math\'ematiques
de Versailles, 45 avenue des Etats-Unis, 78035 Versailles cedex, France}
\footnote{e-mail address: dsp.prof@gmail.com}}

\begin{document}

\thispagestyle{plain}

\maketitle

\begin{abstract}

Let $U$ and $V$ be finite-dimensional vector spaces over a (commutative) field $\K$, and
$\calS$ be a linear subspace of the space $\calL(U,V)$ of all linear operators from $U$ to $V$.
A map $F : \calS \rightarrow V$ is called range-compatible when $F(s) \in \im s$ for all $s \in \calS$.

In a previous work, we have classified all the range-compatible group homomorphisms provided that
$\codim_{\calL(U,V)} \calS \leq 2\dim V-3$, except in the special case when
$\K$ has only two elements and $\codim_{\calL(U,V)} \calS = 2\dim V-3$. In this article,
we give a thorough treatment of that special case. Our results are partly based upon the recent classification of
vector spaces of matrices with rank at most $2$ over $\F_2$.

As an application, we classify the $2$-dimensional non-reflexive operator spaces over any field,
and the affine subspaces of $\Mat_{n,p}(\K)$ with lower-rank $2$ and codimension $3$.
\end{abstract}



\tableofcontents

\section{Introduction}

\subsection{Main definitions and goals}

Let $\K$ be an arbitrary (commutative) field.
We denote by $\Mat_{n,p}(\K)$ the set of matrices with $n$ rows, $p$ columns and entries in $\K$.
Throughout the article, $U$ and $V$ denote finite-dimensional vector spaces over $\K$.
We denote by $\calL(U,V)$ the space of linear operators from $U$ to $V$.
Given a linear subspace $\calS$ of $\calL(U,V)$, the codimension of $\calS$ in $\calL(U,V)$ is simply denoted by
$\codim \calS$.

\begin{Def}
Let $\calS$ be a linear subspace of $\calL(U,V)$, and $F : \calS \rightarrow V$ be a map. \\
We say that $F$ is \textbf{range-compatible} when it satisfies
$$\forall s \in \calS, \quad F(s) \in \im s.$$
We say that $F$ is \textbf{local} when there is a vector
$x \in U$ such that
$$\forall s \in \calS, \quad F(s)=s(x),$$
i.e.\ when $F$ is an evaluation map. In that case, we note that $F$ is linear and range-compatible.
\end{Def}

We adopt similar definitions for maps from a linear subspace of $\Mat_{n,p}(\K)$ to $\K^n$ by
using standard bases to identify $\Mat_{n,p}(\K)$ with $\calL(\K^p,\K^n)$.

\begin{Not}
Let $\calS$ be a linear subspace of $\calL(U,V)$.
The set of all range-compatible linear maps on $\calS$ is a linear subspace of $\calL(\calS,V)$
which we denote by $\calL_\rc(\calS)$;
the subset of all local maps on $\calS$ is a linear subspace of $\calL_\rc(\calS)$ which we denote by
$\calL_\loc(\calS)$.
\end{Not}

\vskip 3mm
Although several authors have independently noticed that every range-compatible linear map on the full space $\calL(U,V)$ is local (this is implicit in \cite{Dieudonne}, for example), the concept of a range-compatible map has only emerged recently.
In \cite{dSPclass}, it was studied as a means to decipher the structure of large vector spaces of matrices with an upper-bound on the rank.
There, the following result was a major key in the generalization to all fields of Atkinson and Lloyd's classification
of such spaces:

\begin{theo}[Lemma 8 of \cite{dSPclass}]\label{classtheo1}
Let $\calS$ be a linear subspace of $\calL(U,V)$
with $\codim \calS \leq \dim V-2$. Then, every range-compatible linear map on $\calS$ is local.
\end{theo}

Theorem \ref{classtheo1} was also used in a sweeping generalization of Dieudonn\'e's theorem on linear bijections
that preserve non-singularity \cite{dSPlargelinpres}.

Besides their links with the theory of spaces of matrices with bounded rank and with linear preservers problems,
range-compatible linear maps are deeply connected, through duality, to the fashionable notion of algebraic reflexivity.
Recall that the \textbf{reflexive closure} of $\calS$, denoted by $\calR(\calS)$, is defined as the space of all linear operators $g : U \rightarrow V$
such that $g(x) \in \calS x$ for all $x \in U$. We say that $\calS$ is (algebraically) reflexive when $\calR(\calS)=\calS$.
In general, the \textbf{reflexivity defect} of $\calS$ is defined as $\dim \calR(\calS)-\dim \calS$.
For $x \in U$, set
$$\widehat{x} : s \in \calS \longmapsto s(x),$$
so that
$$\widehat{\calS}:=\bigl\{\widehat{x} \mid x \in U\bigr\}$$
is a linear subspace of $\calL(\calS,V)$.
Note that $\im \widehat{x}=\calS x$ for all $x \in U$.
From there, the link between the
reflexive closure of $\calS$ and the space of all range-compatible linear maps on $\widehat{\calS}$ is easy to see:
\begin{itemize}
\item Let $F : \widehat{\calS} \rightarrow V$ be a range-compatible linear map. As
$\im \widehat{x}=\calS x$ for all $x \in U$, we see that
the linear map $\check{F} : x \in U \mapsto F(\widehat{x})$ belongs to the reflexive closure of $\calS$.
\item Conversely, let $g \in \calR(\calS)$. For all $x \in U$ such that $\calS x=\{0\}$, we deduce that $g(x)=0$;
therefore, one can find a linear map $G : \widehat{\calS} \rightarrow V$ such that $G(\widehat{x})=g(x)$ for all $x \in U$;
as $g(x) \in \calS x=\im \widehat{x}$ for all $x \in U$, we find that $G$ is range-compatible.
\end{itemize}
With the above, one sees that $F \mapsto \check{F}$ defines an isomorphism from $\calL_\rc(\widehat{\calS})$
to $\calR(\calS)$, and this isomorphism maps $\calL_\loc(\widehat{\calS})$ onto $\calS$.
We deduce the following result:

\begin{prop}
Let $\calS$ be a linear subspace of $\calL(U,V)$.
Then, the quotient spaces $\calL_\rc(\widehat{\calS})/\calL_\loc(\widehat{\calS})$ and $\calR(\calS)/\calS$ are isomorphic through
$F \mapsto \check{F}$. \\
In particular, $\calS$ is reflexive if and only if every range-compatible linear map on $\widehat{\calS}$ is local.
\end{prop}

\begin{Rem}
Similarly, it is easy to demonstrate that the quotient spaces
 $\calR(\widehat{\calS})/\widehat{\calS}$ and $\calL_\rc(\calS)/\calL_\loc(\calS)$ are isomorphic,
 whence $\widehat{\calS}$ is reflexive if and only if every range-compatible linear map on $\calS$ is local.
Besides, the reflexivity defect of $\calS$ equals the codimension of $\calL_\loc(\widehat{\calS})$ in $\calL_\rc(\widehat{\calS})$.
\end{Rem}

In particular, Theorem \ref{classtheo1} yields a sufficient condition for algebraic reflexivity that is based upon the
dimension of the source space of the operator space under consideration (see Theorem 9 of \cite{dSPclass}).

\vskip 3mm
In the first systematic study of range-compatible homomorphisms to date \cite{dSPRC1},
the upper-bound $\dim V-2$ from Theorem \ref{classtheo1} was shown to be non-optimal.
There, the following optimal result was proved:

\begin{theo}\label{maintheolin}
Let $U$ and $V$ be finite-dimensional vector spaces, and $\calS$ be a linear subspace of $\calL(U,V)$
with either $\codim \calS \leq 2 \dim V-3$ if $\# \K>2$, or $\codim \calS \leq 2 \dim V-4$ if $\# \K=2$. \\
Then, every range-compatible linear map on $\calS$ is local.
\end{theo}

In \cite{dSPRC1}, we went as far as to classify, for an arbitrary field with more than $2$ elements, all the
range-compatible \emph{group homomorphisms} on a linear subspace $\calS$ of $\calL(U,V)$ with $\codim \calS \leq 2 \dim V-3$
(see Theorem 1.6 of \cite{dSPRC1}). The main aim of the present article is to examine the case of fields
with two elements. Note already that only the critical case when $\dim \calS=2\dim V-3$ needs to be considered
and that the difficulty does not come from the generalization to group homomorphisms but from the linear maps themselves:
indeed, over $\F_2$ (as is the case over any prime field), a map between vector spaces is a group homomorphism if and only if it is linear.
In that situation, Theorem 1.6 of \cite{dSPRC1} suggests two main special cases when there are non-local range-compatible linear maps on $\calS$
(in the terminology of \cite{dSPRC1}, this happens whenever $\calS$ has Type 2 or Type 3).
A natural question to ask is whether those are the only non-standard cases, or if there are other ones.
Our result is that there is a limited number of other special cases, up to equivalence (five of them, precisely).
The detailed classification in given in Section \ref{maintheoremsection}.

As an application to part of these results and to those of \cite{dSPRC1}, we shall classify all the $2$-dimensional non-reflexive
operator spaces, up to equivalence. In \cite{BracicKuzma}, such a classification was given for all fields with more than $4$ elements;
however, a closer examination of the arguments given there shows that this assumption is mainly there to apply a classification theorem
of Chebotar and \v Semrl \cite{ChebotarSemrl} for locally linearly dependent triples of linear operators, a theorem which is now known
to hold for all fields with more than $2$ elements \cite{dSPLLD2}. However, our own strategy will not be based upon that classification; rather, we
will directly use our own classification of range-compatible linear maps over large operator spaces.

Before we can state our main classification theorem, it is necessary to go through a bit of additional notation.

\subsection{Additional definitions and notation}

In this work, linear hyperplanes are simply called hyperplanes unless specified otherwise.
The entries of matrices are always denoted by small letters, e.g.\ the entry of a matrix $A$ at the $(i,j)$-spot
is denoted by $a_{i,j}$.

We denote by $\Mat_n(\K)$ the algebra of $n$ by $n$ square matrices with entries in $\K$, by
$\GL_n(\K)$ its group of invertible elements, by $\Mats_n(\K)$ its subspace of symmetric matrices, and by
$\Mata_n(\K)$ its subspace of alternating matrices (i.e. skew-symmetric matrices with all diagonal entries zero).
The rank of $M \in \Mat_{n,p}(\K)$ is denoted by $\rk M$. The trace of an endomorphism $u$ of a finite-dimensional vector space
is denoted by $\tr(u)$.

We make the group $\GL_n(\K) \times \GL_p(\K)$ act on the set of linear subspaces of $\Mat_{n,p}(\K)$
by
$$(P,Q).\calV:=P\,\calV\,Q^{-1}.$$
Two linear subspaces of the same orbit will be called \textbf{equivalent}
(this means that they represent, in a change of bases, the same set of linear transformations from a $p$-dimensional vector space to an $n$-dimensional vector space).

We shall consider the bilinear form
$$(u,v) \in \calL(U,V) \times \calL(V,U) \longmapsto \tr(v \circ u).$$
It is non-degenerate on both sides. Throughout the article, orthogonality will always refer to this bilinear form, to the effect that,
given a subset $\calS$ of $\calL(U,V)$, one has
$$\calS^\bot:=\bigl\{v \in \calL(V,U) : \; \forall u \in \calL(U,V), \; \tr(v \circ u)=0\bigr\}.$$
Recall that $(\calS^\bot)^\bot=\calS$ whenever $\calS$ is a linear subspace of $\calL(U,V)$.

Given non-negative integers $m,n,p,q$ and respective subsets $\calA$ and $\calB$ of $\Mat_{m,p}(\K)$ and $\Mat_{n,q}(\K)$,
one sets
$$\calA \vee \calB:=\biggl\{\begin{bmatrix}
A & C \\
[0]_{n \times p} & B
\end{bmatrix} \mid A \in \calA,\; B \in \calB,\; C \in \Mat_{m,q}(\K)\biggr\} \subset \Mat_{m+n,p+q}(\K).$$

Given non-negative integers $n,p,q$ and respective subsets $\calA$ and $\calB$ of $\Mat_{n,p}(\K)$ and $\Mat_{n,q}(\K)$,
one sets
$$\calA \coprod \calB :=\biggl\{\begin{bmatrix}
A & B
\end{bmatrix} \mid A \in \calA, \; B \in \calB\biggr\}.$$

A subspace $\calS$ of $\calL(U,V)$ is called \textbf{reduced} when it satisfies the following conditions:
\begin{enumerate}[(i)]
\item No non-zero vector of $U$ is annihilated by all the operators in $\calS$.
\item The sum of the ranges of the operators in $\calS$ equals $V$.
\end{enumerate}
In the general case, one sets $U_0:= \underset{f \in \calS}{\bigcap} \Ker f$ and
$V_0:=\underset{f \in \calS}{\sum} \im f$, and one sees that every operator $f \in \calS$ induces a linear operator
$$\overline{f} : U/U_0 \rightarrow V_0,$$
and that $\overline{\calS}:=\{\overline{f} \mid f \in \calS\}$ is a reduced subspace of $\calL(U/U_0,V_0)$ called the \textbf{reduced operator space
associated with $\calS$.}

\subsection{The main classification theorem}\label{maintheoremsection}

We have already seen that, on a subspace $\calS$ of $\calL(U,V)$ with $\codim \calS \leq 2\dim V-4$,
every range-compatible linear map is local (see Theorem \ref{maintheolin}). What goes wrong then with fields with two elements and
$\codim \calS=2\dim V-3$? First of all, in \cite{dSPRC1}, the following general result
was proved in which, given vector spaces $V_1$ and $V_2$, a \textbf{root-linear} map
is defined as a group homomorphism $f : V_1 \rightarrow V_2$ such that
$$\forall (\lambda,x)\in \K \times V_1, \; f(\lambda^2 x)=\lambda f(x).$$

\begin{theo}[Corollary 3.4 of \cite{dSPRC1}]
Let $\K$ be a field of characteristic $2$.
Let $r,n,p$ be non-negative integers with $r \geq 2$. Set $\calS:=\Mats_r(\K) \vee \Mat_{n,p}(\K)$.
Then, the group of all range-compatible homomorphisms on $\calS$ is generated by the local maps together with the maps of the form
$$M \longmapsto \begin{bmatrix}
\alpha(m_{1,1}) & \alpha(m_{2,2}) & \cdots & \alpha(m_{r,r}) & 0 & \cdots & 0
\end{bmatrix}^T$$
where $\alpha : \K \rightarrow \K$ is a root-linear form.
\end{theo}

Over $\F_2$, root-linearity is equivalent to linearity, which leads to:

\begin{theo}\label{symmetricF2}
Let $r,n,p$ be non-negative integers with $r \geq 2$. Set $\calS:=\Mats_r(\F_2) \vee \Mat_{n,p}(\F_2)$.
Then, the vector space of all range-compatible linear maps on $\calS$ is generated by the local maps together with
$$M \longmapsto \begin{bmatrix}
m_{1,1} & m_{2,2} & \cdots & m_{r,r} & 0 & \cdots & 0
\end{bmatrix}^T.$$
\end{theo}

To see that the above special case of a range-compatible linear map on $\Mats_r(\F_2) \vee \Mat_{n,p}(\F_2)$
is non-local, note that if there is a vector $X \in \K^{r+p}$ such that
$$\forall M \in \Mats_r(\F_2) \vee \Mat_{n,p}(\F_2), \; MX=\begin{bmatrix}
m_{1,1} & m_{2,2} & \cdots & m_{r,r} & 0 & \cdots & 0
\end{bmatrix}^T,$$
then we find the last $p$ entries of $X$ to be zero by applying the above formula to the matrices of
$\Mats_r(\F_2) \vee \Mat_{n,p}(\F_2)$ with all first $r$ columns zero; then, we show that the
first $r$ entries of $X$ are zero by considering all the matrices of the form
$M=\begin{bmatrix}
A & [0]_{r \times p} \\
[0]_{n \times r} & [0]_{n \times p}
\end{bmatrix}$ with $A \in \Mata_r(\K)$; thus $X=0$, which is absurd.

In particular, if a linear subspace $\calS$ of $\calL(U,V)$ is represented by $\Mats_2(\F_2) \vee \Mat_{n,p}(\F_2)$ or by
$\Mats_3(\F_2) \coprod \Mat_{3,p}(\K)$ for some pair $(n,p)$ of non-negative integers, then
there is a non-local range-compatible linear map on it.

\vskip 3mm
There are other examples. In order to discuss them, some additional notation is necessary:

\begin{Not}
We define:
\begin{itemize}
\item $\calV_2:=\Biggl\{
\begin{bmatrix}
a & b \\
b & c \\
c & 0
\end{bmatrix}\mid (a,b,c)\in (\F_2)^3\Biggr\}$;
\item $\calG_3:=\Biggl\{
\begin{bmatrix}
a & c & b \\
0 & b+c & e \\
b & d & f
\end{bmatrix} \mid (a,b,c,d,e,f)\in (\F_2)^6\Biggr\}$;
\item $\calH_3:=\Biggl\{
\begin{bmatrix}
a & b & c \\
b & d & f \\
c & e & b+c+d
\end{bmatrix}\mid (a,b,c,d,e,f)\in (\F_2)^6\Biggr\}$;
\item $\calI_3:=\Biggl\{
\begin{bmatrix}
a & d & e \\
b & c & f \\
c & a & a+c+e+f
\end{bmatrix}\mid (a,b,c,d,e,f)\in (\F_2)^6\Biggr\}$;
\item $\calH_4:=\Biggl\{
\begin{bmatrix}
a & b+c & f & h \\
b & d & a+c & i \\
c & e & g & a+b
\end{bmatrix}\mid (a,b,c,d,e,f,g,h,i)\in (\F_2)^9\Biggr\}$.
\end{itemize}
\end{Not}

We note that each of those spaces has codimension $3$ in the full matrix space it is naturally included in.

\begin{Def}
Let $\calS$ be a linear subspace of $\calL(U,V)$.
We say that $\calS$ has \textbf{Type $i$} when, in well-chosen bases of
$U$ and $V$, it is represented by the matrix space featured in the corresponding line of the following array.

\begin{center}
\begin{tabular}{| c | l |}
\hline
Type & Matrix space representing $\calS$ in well-chosen bases of $U$ and $V$ \\
\hline
\hline
$1$ & $\Mats_2(\F_2) \vee \Mat_{n,p}(\F_2)$, with $n \geq 0$ and $p \geq 0$. \\
\hline
$2$ & $\Mats_3(\F_2) \coprod \Mat_{3,p}(\F_2)$, with $p \geq 0$. \\
\hline
$3$ & $\calV_2 \vee \Mat_{n,p}(\F_2)$, with $n \geq 0$ and $p \geq 0$. \\
\hline
$4$  & $\calG_3 \coprod \Mat_{3,p}(\F_2)$, with $p \geq 0$. \\
\hline
$5$ & $\calH_3 \coprod \Mat_{3,p}(\F_2)$, with $p \geq 0$. \\
\hline
$6$ & $\calI_3 \coprod \Mat_{3,p}(\F_2)$, with $p \geq 0$. \\
\hline
$7$ & $\calH_4 \coprod \Mat_{3,p}(\F_2)$, with $p \geq 0$. \\
\hline
\end{tabular}
\end{center}
\end{Def}

Note that in the above cases $\calS$ has codimension $2\dim V-3$ in $\calL(U,V)$.
Spaces of Type 1 and 2 in the above classification correspond, respectively, to spaces
of Type 2 and 3 from \cite{dSPRC1}.

\begin{theo}\label{classRCF2}
Assume that $\K=\F_2$.
Let $\calS$ be a linear subspace of $\calL(U,V)$ with codimension $2\dim V-3$, and which has none of Types 1 to 7.
Then, every range-compatible linear map on $\calS$ is local.
\end{theo}

We have already described the range-compatible linear maps on spaces of Type 1 or 2. In the following theorem,
we recall these results and describe the remaining five cases:

\begin{theo}\label{specialRCF2}
Assume that $\K=\F_2$.
Let $\calS$ be a linear subspace of $\calL(U,V)$ that has one of Types 1 to 7.
Then, $\calL_\loc(\calS,V)$ has codimension $1$ in $\calL_\rc(\calS,V)$.
In the following array, we give a non-local range-compatible linear map from each special type of space:
\begin{center}
\begin{tabular}{| c | c | l |}
\hline
Type & Matrix space & Example of a non-local range-compatible linear map \\
\hline
\hline
$1$ & $\Mats_2(\F_2) \vee \Mat_{n,p}(\F_2)$ & $M \mapsto \begin{bmatrix}
m_{1,1} \\
m_{2,2} \\
[0]_{n \times 1}
\end{bmatrix}$ \\
\hline
$2$ & $\Mats_3(\F_2) \coprod \Mat_{3,p}(\F_2)$ & $M \mapsto \begin{bmatrix}
m_{1,1} \\
m_{2,2} \\
m_{3,3}
\end{bmatrix}$ \\
\hline
$3$ & $\calV_2 \vee \Mat_{n,p}(\F_2)$ & $M \mapsto \begin{bmatrix}
0 \\
m_{2,1}+m_{2,2} \\
0 \\
[0]_{n \times 1}
\end{bmatrix}$ \\
\hline
$4$ & $\calG_3 \coprod \Mat_{3,p}(\F_2)$ & $M \mapsto \begin{bmatrix}
m_{1,1}+m_{1,3} \\
0 \\
0
\end{bmatrix}$ \\
\hline
$5$ & $\calH_3 \coprod \Mat_{3,p}(\F_2)$ & $M \mapsto \begin{bmatrix}
m_{1,1} \\
m_{2,2} \\
m_{3,3}
\end{bmatrix}$ \\
\hline
$6$ & $\calI_3 \coprod \Mat_{3,p}(\F_2)$ &
$M \mapsto \begin{bmatrix}
0 \\
0 \\
m_{1,1}+m_{3,1}
\end{bmatrix}$ \\
\hline
$7$ & $\calH_4 \coprod \Mat_{3,p}(\F_2)$ &
$M \mapsto \bigl(m_{1,1}+m_{2,1}+m_{3,1}\bigr)\begin{bmatrix}
1 \\
1 \\
1
\end{bmatrix}$ \\
\hline
\end{tabular}
\end{center}
\end{theo}

Finally, the above special spaces are pairwise inequivalent:

\begin{theo}\label{inequivalencetheorem}
Given distinct integers $i$ and $j$ in $\lcro 1,7\rcro$, no space can have both Types $i$ and $j$.
\end{theo}

\subsection{Strategy of proof, and structure of the article}

Our proof of the above results is split into two independent blocks.
In the first one (Section \ref{specialtypesection}), we establish Theorems \ref{specialRCF2} and \ref{inequivalencetheorem}.
In the second one (Sections \ref{mainproofsec1} and \ref{mainproofsec2}), we prove Theorem \ref{classRCF2}.

For both proofs, we will need a lot of basic results that were developed in \cite{dSPRC1},
in particular quotient space techniques. The main idea is that if
$F : \calS \rightarrow V$ is a range-compatible linear map and $y$ is a non-zero vector of $V$,
then $F$ induces a range-compatible linear map
$$(F \modu y) : (\calS \modu \K y) \longrightarrow V/\K y,$$
where $\calS \modu \K y$ denotes the space of operators from $U$ to $V/\K y$ that is naturally associated with $\calS$.
If the codimension of $\calS \modu \K y$ is small enough, then we can use an induction on the dimension of $V$
to recover precious information on $F$. The vectors $y$ for which we can warrant that the codimension of
$\calS \modu \K y$ is small enough will be called the \emph{$\calS$-adapted vectors}.
A very important lemma (Lemma \ref{3vectorslemma}) that was proved in \cite{dSPRC1}
states that if $\codim \calS \leq 2 \dim V-3$ and if we can find three linearly independent $\calS$-adapted vectors in $V$,
then every range-compatible linear map on $\calS$ is local.
On the other hand, having too few $\calS$-adapted vectors in $V$ translates into rank properties of
the dual space $\widehat{\calS^\bot}$, and in some instances it is then possible to show that every operator in
$\widehat{\calS^\bot}$ has rank at most 2 (this was essentially the strategy in \cite{dSPRC1}).
In those situations, we shall appeal to the recent classification of spaces of matrices with rank at most $2$ over $\F_2$
\cite{dSPprimitiveF2} to uncover the structure of $\calS$.

In Section \ref{toolssection}, we shall recall all the useful technical results
on range-compatible linear maps that were already established in \cite{dSPRC1},
and then we shall gather the results from \cite{dSPprimitiveF2} that we will use in the proof of Theorem \ref{classRCF2}.

The last two sections (Sections \ref{dim2reflexive} and \ref{affinespacesection})
are devoted to applications of Theorems \ref{classRCF2} and \ref{specialRCF2},
first to the classification of non-reflexive $2$-dimensional spaces of operators, and
then to the one of large affine spaces in which no matrix has rank less than $2$.

\section{Main tools}\label{toolssection}

Here, we review some basic results that were proved in \cite{dSPRC1}.
Throughout the section, $\K$ denotes the field $\F_2$.

\subsection{Range-compatible linear maps in specific cases}

The first two lemmas are the most basic results on range-compatible linear maps.

\begin{lemma}[Corollary 2.2 in \cite{dSPRC1}]\label{dimU=1}
Assume that $\dim U=1$. Let $\calS$ be a linear subspace of $\calL(U,V)$.
Then, every range-compatible linear map on $\calS$ is local.
\end{lemma}

\begin{lemma}[Proposition 2.5 in \cite{dSPRC1}]\label{espacetotal}
Every range-compatible linear map on $\calL(U,V)$ is local.
\end{lemma}

\subsection{Embedding and splitting techniques}

Here, we recall two basic techniques for dealing with range-compatible linear maps on matrix spaces.
The first one is obvious. The second one is Lemma 2.4 in \cite{dSPRC1}.

\begin{lemma}[Embedding Lemma]\label{embeddinglemma}
Let $\calS$ be a linear subspace of $\Mat_{n,p}(\K)$, and let $n'$ be a non-negative integer.
Consider the space $\calS' \subset \Mat_{n+n',p}(\K)$ of all matrices of the form $\begin{bmatrix}
M \\
[0]_{n' \times p}
\end{bmatrix}$ with $M \in \calS$, and let $F' : \calS' \rightarrow \K^{n+n'}$ be a range-compatible linear map. \\
Then, there is a range-compatible linear map $F : \calS \rightarrow \K^n$ such that
$$\forall M \in \calS, \; F'\Biggl(\begin{bmatrix}
M \\
[0]_{n' \times p}
\end{bmatrix}\Biggr)=\begin{bmatrix}
F(M) \\
[0]_{n' \times 1}
\end{bmatrix}.$$
\end{lemma}

\begin{lemma}[Splitting Lemma]\label{splittinglemma}
Let $n,p,q$ be non-negative integers, and $\calA$ and $\calB$ be linear subspaces, respectively, of $\Mat_{n,p}(\K)$ and $\Mat_{n,q}(\K)$.

Given maps $f : \calA \rightarrow \K^n$ and $g : \calB \rightarrow \K^n$,
set
$$f \coprod g : \begin{bmatrix}
A & B
\end{bmatrix} \in \calA \coprod \calB \longmapsto f(A)+g(B).$$
Then:
\begin{enumerate}[(a)]
\item The linear maps from $\calA \coprod \calB$ to $\K^n$
are the maps of the form $f \coprod g$, where $f\in \calL(\calA,\K^n)$ and $g \in \calL(\calB,\K^n)$.
Moreover, every linear map from $\calA \coprod \calB$ to $\K^n$ may be expressed in a unique fashion as $f \coprod g$.
\item Given $f \in \calL(\calA,\K^n)$ and $g \in \calL(\calB,\K^n)$, the map
$f \coprod g$ is range-compatible (respectively, local) if and only if $f$ and $g$ are range-compatible (respectively, local).
\end{enumerate}
\end{lemma}

\subsection{The projection lemma}

Now, we come to the projection technique: this cornerstone of the proof of Theorem 1.6 of \cite{dSPRC1}
will remain our basic tool for proving Theorem \ref{classRCF2} by induction on the dimension of $V$:

\begin{lemma}[Projection Lemma, Lemma 2.6 of \cite{dSPRC1}]
Let $\calS$ be a linear subspace of $\calL(U,V)$ and $V_0$ be a linear subspace of $V$.
Let $F : \calS \rightarrow V$ be a range-compatible linear map.
Denote by $\pi : V \twoheadrightarrow V/V_0$ the canonical projection, and
by $\calS \modu V_0$ the space of all linear maps of the form $\pi \circ s$ with $s \in \calS$.
Then, there is a unique range-compatible linear map
$$(F \modu V_0) : \calS \modu V_0 \rightarrow V/V_0$$
such that
$$\forall s \in \calS, \; (F \modu V_0)(\pi \circ s)=\pi(F(s)),$$
i.e.\ the following diagram is commutative:
$$\xymatrix{
\calS \ar[rr]^F \ar[d]_{s \mapsto \pi \circ s} & & V \ar[d]^\pi \\
\calS \modu V_0 \ar[rr]_{F \modu V_0} & & V/V_0.
}$$
In particular, given a non-zero vector $y \in V$, one denotes by $F \modu y$ the projected map $F \modu \K y$,
and by $\calS \modu y$ the operator space $\calS \modu \K y$.
\end{lemma}

In terms of matrices, the special case when $V_0$ is a linear hyperplane of $V$ has the following interpretation:

\begin{lemma}\label{ligneparligne}
Let $\calS$ be a linear subspace of $\Mat_{n,p}(\K)$, and $F$ be a range-compatible linear map on $\calS$.
For $i \in \lcro 1,n\rcro$ and $M \in \calS$, denote by $R_i(M)$ the $i$-th row of $M$.
Then, there are linear forms $F_1,\dots,F_n$, respectively, on $R_1(\calS),\dots,R_n(\calS)$, such that
$$F : \begin{bmatrix}
L_1 \\
\vdots \\
L_n
\end{bmatrix} \longmapsto \begin{bmatrix}
F_1(L_1) \\
\vdots \\
F_n(L_n)
\end{bmatrix}.$$
\end{lemma}

\subsection{Adapted vectors}

\begin{Def}
Let $\calS$ be a linear subspace of $\calL(U,V)$.
A non-zero vector $y \in V$ is called \textbf{$\calS$-adapted} whenever
$$\codim (\calS \modu y) \leq 2 (\dim V-1)-3.$$
\end{Def}

In general, by duality one finds
$$\codim (\calS \modu y)=\codim \calS-\dim \calS^\bot y.$$
Therefore, in the special case when $\codim \calS=2 \dim V-3$, the vector $y$ is $\calS$-adapted if and only if
$\dim \calS^\bot y \geq 2$.

In \cite{dSPRC1}, we have proved the following result, which helps obtain many adapted vectors
(this combines \cite[Lemma 4.1]{dSPRC1} with \cite[Lemma 6.1]{dSPRC1}):

\begin{lemma}[Adapted vectors lemma]\label{adaptedvectorslemma}
Let $\calS$ be a linear subspace of $\calL(U,V)$ with $\codim \calS \leq 2\dim V-3$.
Then, either the set of all non-$\calS$-adapted vectors is included in a hyperplane of $V$
or every range-compatible linear map on $\calS$ is local.
\end{lemma}

\subsection{Sufficient conditions for localness}

In \cite{dSPRC1}, the following result was a major key to the proof of Theorem \ref{classtheo1}; it will also be
very important in the present study:

\begin{lemma}[Lemma 4.2 of \cite{dSPRC1}]\label{3vectorslemma}
Let $\calS$ be a linear subspace of $\calL(U,V)$ with $\codim \calS \leq 2\dim V-3$.
Let $F : \calS \rightarrow V$ be a range-compatible group homomorphism.
Assume that there are linearly independent vectors $y_1$, $y_2$ and $y_3$ of $V$
such that $F\modu y_1$, $F\modu y_2$, $F\modu y_3$ are all local.
Then, $F$ is local.
\end{lemma}

In addition, we shall use the following known result:

\begin{lemma}[Proposition 2.9 of \cite{dSPRC1}]\label{vecteurdim1}
Let $\calS$ be a linear subspace of $\calL(U,V)$ with $\codim \calS \leq 2\dim V-3$.
Assume that there is a non-zero vector $x$ of $U$ such that $\dim \calS x \leq 1$.
Then, every range-compatible linear map on $\calS$ is local.
\end{lemma}

\subsection{A covering lemma}

The following lemma on coverings of a vector space by linear subspaces, which is proved in \cite{dSPfeweigenvalues}, will be used in a few instances.

\begin{lemma}[Lemma 2.3 of \cite{dSPfeweigenvalues}]\label{coveringlemma}
Let $p$ be a positive integer, $E$ be an $n$-dimensional vector space over a field with more than $p$ elements, and $(E_i)_{i \in I}$
be a family of $(n-1)p+1$ linear subspaces of $E$ in which exactly $p+1$ vector spaces have dimension $n-1$ and, for
all $k \in \lcro 1,n-2\rcro$, exactly $p$ vector spaces have dimension $k$.
Then, $E$ is not included in $\underset{i \in I}{\bigcup} E_i$.
\end{lemma}

\subsection{A lemma on quadratic forms over $\F_2$}

The following lemma was proved in \cite{dSPRC1}:

\begin{lemma}[Lemma 5.2 of \cite{dSPRC1}]\label{quadformlemma}
Let $q$ be a non-zero quadratic form on an $n$-dimensional vector space $E$ over $\F_2$.
Then, $q^{-1}\{1\}$ is not included in an $(n-2)$-dimensional linear subspace of $E$.
\end{lemma}

\subsection{Primitive spaces of matrices with upper-rank $2$ over $\F_2$}\label{primitivesection}

Here, we review some results from \cite{dSPprimitiveF2}.

The \textbf{upper-rank} of a linear subspace $\calV$ of $\Mat_{n,p}(\K)$ is defined as the maximal rank for a matrix in $\calV$:
we denote it by $\urk(\calV)$.

A linear subspace $\calV$ of $\Mat_{n,p}(\K)$ with upper-rank $r$ is called \textbf{primitive} when it
is reduced and satisfies the two extra conditions below:
\begin{enumerate}[(i)]
\item $\calV$ is not equivalent to a space $\calT$ of matrices of the form
$M=\begin{bmatrix}
H(M) & [?]_{n \times 1}
\end{bmatrix}$ where $\urk H(\calT) \leq r-1$;
\item $\calV$ is not equivalent to a space $\calT$ of matrices
of the form
$M=\begin{bmatrix}
H(M) \\
[?]_{1 \times p}
\end{bmatrix}$ where $\urk H(\calT) \leq r-1$.
\end{enumerate}
Note that this definition is invariant under replacing $\calV$ with an equivalent subspace.

The following result is a consequence of Proposition 1.1 of \cite{dSPprimitiveF2}
and of the standard classification of spaces with upper-rank $1$:

\begin{prop}\label{nonprimitiveprop}
Let $\calV$ be a non-primitive reduced linear subspace of $\Mat_{n,p}(\F_2)$ with upper-rank at most $2$.
Then, either $n=2$, or $p=2$, or $\calV$ is equivalent to a subspace of the space of all matrices of the form
$$\begin{bmatrix}
? & [?]_{1 \times (p-1)} \\
[?]_{(n-1) \times 1} & [0]_{(n-1) \times (p-1)}
\end{bmatrix}.$$
\end{prop}

We shall also need the following two results, both of which come from Theorem 1.5 of \cite{dSPprimitiveF2}:

\begin{prop}\label{primitivedimprop}
Let $\calV$ be a primitive linear subspace of $\Mat_{n,p}(\F_2)$ with upper-rank $2$.
Then, $n=p=3$.
\end{prop}

\begin{prop}\label{primitivetypeprop1}
Let $\calV$ be a primitive linear subspace of $\Mat_3(\F_2)$ with upper-rank $2$.
Assume that $\dim \calV=3$ and that there is no vector $x \in (\F_2)^3$
such that $\dim \calV x=1$. Then, $\calV$ is equivalent to $\Mata_3(\F_2)$ or to the space
$$\calU_3(\F_2):=\Biggl\{\begin{bmatrix}
0 & a & a+c \\
a & 0 & b \\
a+b & c & 0
\end{bmatrix} \mid (a,b,c)\in (\F_2)^3\Biggr\}.$$
Conversely, $\Mata_3(\F_2)$ and $\calU_3(\F_2)$ are $3$-dimensional primitive subspaces of $\Mat_3(\F_2)$
in which every non-zero matrix has rank $2$.
\end{prop}

Let us explain how Proposition \ref{primitivetypeprop1} is derived from Theorem 1.5 of \cite{dSPprimitiveF2}:
Combining the assumption that no vector $x \in (\F_2)^3$ satisfies $\dim \calV x=1$ and the one that
$\calV$ is reduced, we obtain that $\calV$ cannot be equivalent to a subspace of upper-triangular matrices,
and in particular $\calV$ cannot be equivalent to a subspace of the space denoted by $\calJ_3(\F_2)$ in \cite{dSPprimitiveF2}.
On the other hand, as $\calV$ has dimension $3$ it cannot be equivalent to the space denoted by $\calV_3(\F_2)$ in \cite{dSPprimitiveF2},
which only leaves open the possibility that $\calV$ is equivalent to $\Mata_3(\F_2)$ or to $\calU_3(\F_2)$.

At some point we will need the following result, which follows directly from Lemma 3.1 and Proposition 4.2 of \cite{dSPprimitiveF2}.

\begin{prop}\label{primitivetypeprop2}
Let $\calV$ be a $3$-dimensional primitive linear subspace of $\Mat_3(\F_2)$ with upper-rank $2$.
Assume that there is a vector $x \in (\F_2)^3$ such that $\dim \calV x=1$.
Then, $\calV$ is equivalent to one of the following four spaces:
$$\calM_1:=\Biggl\{\begin{bmatrix}
a & 0 & c  \\
0 & a+b & 0 \\
0 & 0 & b
\end{bmatrix} \mid (a,b,c) \in (\F_2)^3\Biggr\}, \quad
\calM_2:=\Biggl\{\begin{bmatrix}
a & c & 0 \\
0 & a+b & a \\
0 & 0 & b
\end{bmatrix}
\mid (a,b,c) \in (\F_2)^3\Biggr\}$$
$$\calM_3:=\Biggl\{\begin{bmatrix}
a & b & 0  \\
0 & a+b & c \\
0 & 0 & b
\end{bmatrix} \mid (a,b,c) \in (\F_2)^3\Biggr\}, \quad \calM_4:=\Biggl\{\begin{bmatrix}
a & c & 0  \\
0 & a+b & c \\
0 & 0 & b
\end{bmatrix} \mid (a,b,c) \in (\F_2)^3\Biggr\}.$$
Conversely, the $\calM_i$ spaces all satisfy the given conditions.
\end{prop}

In order to differentiate between the above special types of spaces, the following result from \cite{dSPprimitiveF2} will also be useful.

\begin{prop}\label{primitivetypeprop3}
Let $\calV$ be a linear subspace of $\Mat_3(\F_2)$. The following conditions are pairwise incompatible:
\begin{enumerate}[(i)]
\item $\calV$ is equivalent to a linear subspace of the space $\calJ_3(\F_2)$ of all upper-triangular matrices with trace $0$;
\item $\calV$ is equivalent to $\Mata_3(\F_2)$;
\item $\calV$ is equivalent to $\calU_3(\F_2)$.
\end{enumerate}
\end{prop}

\section{Spaces of special type and their range-compatible linear maps}\label{specialtypesection}

In Theorem \ref{specialRCF2}, the results on spaces of Type 1 or 2 follow directly from Theorem \ref{symmetricF2}.
In this section, we examine the remaining five cases.
In order to do so, we tackle each case separately. Using the splitting lemma, it is obvious that only
the five following matrix spaces need to be considered: $\calV_2$, $\calG_3$, $\calH_3$, $\calI_3$ and $\calH_4$.
Throughout the section, we set $\K:=\F_2$ and we denote by $(e_1,e_2,e_3)$ the standard basis of $\K^3$.

\subsection{Spaces of Type 3}

Let us describe the range-compatible linear maps on $\calV_2$.
Let $F : \calV_2 \rightarrow \K^3$ be a range-compatible linear map.
Applying Theorem \ref{symmetricF2} to $F \modu e_3$ yields that $F \modu e_3$ is the sum of a local map and,
for some $\varepsilon \in \F_2$, of the map represented in
the standard basis of $\K^2$ and in $(\overline{e_1},\overline{e_2})$ by $\begin{bmatrix}
a & b \\
b & c
\end{bmatrix} \mapsto \varepsilon\begin{bmatrix}
a \\
c
\end{bmatrix}$. Then, as we lose no generality in subtracting a local map from $F$, we see that no generality is lost in assuming that
$$F : \begin{bmatrix}
a & b \\
b & c \\
c & 0
\end{bmatrix} \longmapsto \begin{bmatrix}
\varepsilon a \\
\varepsilon c \\
?
\end{bmatrix}.$$
Applying Lemma \ref{ligneparligne} to the third row, we obtain another scalar $\eta \in \K$ such that
$$F : \begin{bmatrix}
a & b \\
b & c \\
c & 0
\end{bmatrix} \longmapsto \begin{bmatrix}
\varepsilon a \\
\varepsilon c \\
\eta c
\end{bmatrix}.$$
Then, for all $(a,b,c)\in \K^3$, we deduce that
$$0=\begin{vmatrix}
a & b & \varepsilon a \\
b & c & \varepsilon c \\
c & 0 & \eta c
\end{vmatrix}=(\eta+\varepsilon)(a+b)c.$$
It follows that $\eta=\varepsilon$. Thus, either $F$ is local or
$$F : \begin{bmatrix}
a & b \\
b & c \\
c & 0
\end{bmatrix} \longmapsto
\begin{bmatrix}
a \\
c \\
c
\end{bmatrix}.$$
In the latter case, adding the local map $M \mapsto M \times \begin{bmatrix}
1 \\
0
\end{bmatrix}$ to $F$ yields
$$G : \begin{bmatrix}
a & b \\
b & c \\
c & 0
\end{bmatrix} \longmapsto
\begin{bmatrix}
0 \\
b+c \\
0
\end{bmatrix}.$$
Thus, in any case we have proved that every non-local range-compatible linear map on $\calV_2$ is the sum of a local map
with $G$.

\vskip 3mm
Conversely, let us prove that $G$ is range-compatible and non-local.
Let $M=\begin{bmatrix}
a & b \\
b & c \\
c & 0
\end{bmatrix} \in \calV_2$. If $b=0$, then $G(M)$ is the second column of $M$. If $b=c$, then $G(M)=0$.
The last remaining case is the one when $b=1$ and $c=0$, in which $G(M)=M \times \begin{bmatrix}
1 \\
a
\end{bmatrix}$. Therefore, $G(M) \in \im M$ in any case.
On the other hand, it is easily seen from the first two rows that $G$ is non-local.

We conclude that
$$\calL_\rc(\calV_2)=\calL_\loc(\calV_2) \oplus \K G.$$
Using the Splitting Lemma, this settles the case of spaces of Type 3 in Theorem \ref{specialRCF2}.

\subsection{Spaces of Type 4}

Let $F : \calG_3 \rightarrow \K^3$ be a range-compatible linear map.
Seeing that $\calS \modu e_1$ is equivalent to $\K \vee \Mat_{1,2}(\K)$, we deduce from Lemma \ref{vecteurdim1} that
$F \modu e_1$ is local. Then, no generality is lost in assuming that $F \modu e_1=0$.
Noting that $\calS \modu e_2$ is deduced from $\Mats_2(\K) \coprod \K^2$ through a simple permutation of columns,
we use Theorem \ref{symmetricF2} to obtain scalars
$\alpha,\beta,\gamma,\delta$ such that
$$F : \begin{bmatrix}
a & c & b \\
0 & b+c & e \\
b & d & f
\end{bmatrix} \longmapsto \alpha \begin{bmatrix}
a \\
? \\
b
\end{bmatrix}+\beta \begin{bmatrix}
c \\
? \\
d
\end{bmatrix}+\gamma \begin{bmatrix}
b \\
? \\
f
\end{bmatrix}+\delta \begin{bmatrix}
a \\
? \\
f
\end{bmatrix}.$$
Since $F \modu e_1=0$, we deduce that $\alpha b+\beta d+(\gamma+\delta) f=0$ for all $(b,d,f)\in \K^3$, whence $\alpha=\beta=0$ and $\gamma=\delta$.
It follows that $F=\gamma G$ where
$$G : \begin{bmatrix}
a & c & b \\
0 & b+c & e \\
b & d & f
\end{bmatrix} \longmapsto \begin{bmatrix}
a+b \\
0 \\
0
\end{bmatrix}.$$
Conversely, let us prove that $G$ is range-compatible and non-local.
Let $M=\begin{bmatrix}
a & c & b \\
0 & b+c & e \\
b & d & f
\end{bmatrix} \in \calG_3$. If $a=b$, we have $G(M)=0$. If $a=1$ and $b=0$, then $G(M)$ is the first column of $M$. \\
Assume now that $a=0$ and $b=1$. If $c=0$, then $M$ is invertible, whence $G(M)$ belongs to its column space.
Finally if $c=1$, then one sees that $G(M)=M \times \begin{bmatrix}
d \\
1 \\
0
\end{bmatrix}$. Therefore, $G(M) \in \im M$ in any case.

If $G$ were local, then we would have $G=0$ as seen from the last row, which is obviously false.

We conclude that
$$\calL_\rc(\calG_3)=\calL_\loc(\calG_3) \oplus \K G.$$
Using the Splitting Lemma, this settles the case of spaces of Type 4 in Theorem \ref{specialRCF2}.

\subsection{Spaces of Type 5}

Let $F : \calH_3 \rightarrow \K^3$ be a range-compatible linear map.
We note that $\calH_3 \modu e_3$ has Type 1.
Thus, subtracting a local map if necessary, we see that no generality is lost in assuming that
there is some $\varepsilon \in \K$ such that
$$F : \begin{bmatrix}
a & b & c \\
b & d & f \\
c & e & b+c+d
\end{bmatrix} \longmapsto
\begin{bmatrix}
\varepsilon a \\
\varepsilon d \\
?
\end{bmatrix}.$$
Then, we find a triple $(\lambda,\mu,\nu) \in \K^3$ such that
$$F : \begin{bmatrix}
a & b & c \\
b & d & f \\
c & e & b+c+d
\end{bmatrix} \longmapsto
\begin{bmatrix}
\varepsilon a \\
\varepsilon d \\
\lambda c+\mu e+\nu (b+c+d)
\end{bmatrix}.$$
It follows that $F \modu e_2$ is represented by
$$\begin{bmatrix}
a & b & c \\
c & e & g
\end{bmatrix} \longmapsto \begin{bmatrix}
\varepsilon a \\
\lambda c+\mu e+\nu g
\end{bmatrix}.$$
With $(a,b,c,e,g)=(0,1,0,1,0)$, we obtain $\mu=0$.
On the other hand, $F \modu e_1$ is represented by
$$\begin{bmatrix}
b & d & f \\
c & e & b+c+d
\end{bmatrix} \longmapsto
\begin{bmatrix}
\varepsilon d \\
\lambda c+\nu (b+c+d)
\end{bmatrix}.$$
With $(b,c,d,e,f)=(1,1,0,0,0)$, we deduce that $\lambda=0$. Finally, with
$(b,c,d,e,f)=(1,1,1,1,1)$, we conclude that $\nu=\varepsilon$. Thus,
$F=\varepsilon G$, where
$$G : \begin{bmatrix}
a & b & c \\
b & d & f \\
c & e & b+c+d
\end{bmatrix} \longmapsto
\begin{bmatrix}
a \\
d \\
b+c+d
\end{bmatrix}.$$
Conversely, let us prove that $G$ is non-local and range-compatible. From the first row, we see that if $G$ were local,
then we would have $G : M \mapsto M e_1$, which is obviously false.
Now, let $M=\begin{bmatrix}
a & b & c \\
b & d & f \\
c & e & b+c+d
\end{bmatrix} \in \calH_3$ be with $G(M) \neq 0$.
We use a \emph{reductio ad absurdum}, by assuming that $G(M)$ is not in the column space of $M$.
In particular, $M$ must be singular. By Theorem \ref{symmetricF2}, $M$ cannot be symmetric, whence $e \neq f$.
Noting that $\calH_3$ is invariant under conjugating by $P:=\begin{bmatrix}
1 & 0 & 0 \\
0 & 0 & 1 \\
0 & 1 & 0
\end{bmatrix}$, and noting that $G(PMP^{-1})=PG(M)$, we see that no generality is lost in assuming that
$e=0$ and $f=1$.

As $G(M)$ is not the first column of $M$, we have $b \neq d$, whence $d=b+1$.
Then, $M=\begin{bmatrix}
a & b & c \\
b & b+1 & 1 \\
c & 0 & c+1
\end{bmatrix}$ and $G(M)=\begin{bmatrix}
a \\
b+1 \\
c+1
\end{bmatrix}$. If $c=0$, one finds that $G(M)$ is the sum of the first and third columns of $M$.
Thus, $c=1$. Then, one finds $\det M=1$, contradicting a previous result. We conclude that $G$ is range-compatible.

Therefore,
$$\calL_\rc(\calH_3)=\calL_\loc(\calH_3) \oplus \K G.$$
Using the Splitting Lemma, the case of spaces of Type 5 in Theorem \ref{specialRCF2} ensues.

\subsection{Spaces of Type 6}

Let $F : \calI_3 \rightarrow \K^3$ be a range-compatible linear map.
Note that $\calI_3 \modu e_3$ is the space of all linear maps from $\K^3$ to $\K^3/\K e_3$, whence
$F \modu e_3$ is local. We deduce that no generality is lost in assuming that $F$ maps every matrix of
$\calI_3$ into $\K e_3$.
Thus, we have scalars $\lambda,\mu,\nu$ such that
$$F : \begin{bmatrix}
a & d & e \\
b & c & f \\
c & a & a+c+e+f \\
\end{bmatrix} \longmapsto \begin{bmatrix}
0 \\
0 \\
\lambda a+\mu c+\nu (e+f)
\end{bmatrix}.$$
Taking $(a,b,c,d,e,f)=(0,0,0,0,0,1)$, we find a matrix whose column space is spanned by
$\begin{bmatrix}
0 \\
1 \\
1
\end{bmatrix}$, whence $\nu=0$. \\
Taking $(a,b,c,d,e,f)=(1,1,1,1,0,0)$, we find  a matrix whose column space is spanned by
$\begin{bmatrix}
1 \\
1 \\
1
\end{bmatrix}$, whence $\lambda+\mu=0$. Therefore, $F=\lambda G$, where
$$G : \begin{bmatrix}
a & d & e \\
b & c & f \\
c & a & a+c+e+f
\end{bmatrix} \longmapsto \begin{bmatrix}
0 \\
0 \\
a+c
\end{bmatrix}.$$

Conversely, let us prove that $G$ is range-compatible and non-local.
If there were a vector $X \in \K^3$ such that $G : M \mapsto MX$,
then $X=0$ by considering the first row, whence $G=0$, which is obviously false. Therefore, $G$ is non-local.

Now, let $M=\begin{bmatrix}
a & d & e \\
b & c & f \\
c & a & a+c+e+f \\
\end{bmatrix}\in \calI_3$ be such that $G(M) \neq 0$. Then, $a+c=1$, whence $(a,c)=(1,0)$ or $(a,c)=(0,1)$. \\
Note that none of the first two columns of $M$ is zero, and that there are different, judging from the last row.
Therefore, they are linearly independent. If $bd=0$, then
we see that
$$\begin{vmatrix}
a & d & 0 \\
b & c & 0 \\
c & a & 1
\end{vmatrix}=ac+bd=0,$$
which yields that $G(M)$ is a linear combination of the first two columns of $M$. Assume now that $bd=1$, so that $b=d=1$.
Then,
$$\det(M)=ac(1+e+f)+cf+ae+ec^2+fa^2+(1+e+f)=f(a+c+1)+e(a+c+1)+1=1,$$
whence $G(M) \in \im M$. Therefore, $G$ is range-compatible.

We conclude that
$$\calL_\rc(\calI_3)=\calL_\loc(\calI_3) \oplus \K G.$$
Using the Splitting Lemma, the case of spaces of Type 6 in Theorem \ref{specialRCF2} ensues.

\subsection{Spaces of Type 7}

Let $F : \calH_4 \rightarrow \K^3$ be a range-compatible linear map.
For $y:=e_1+e_2+e_3$, we compute that $\calH_4^\bot y=\Bigl\{ \begin{bmatrix}
0 & a & b & c
\end{bmatrix}^T \mid (a,b,c)\in \K^3\Bigr\}$ has dimension $3$, whence $\codim (\calH_4 \modu y)=0$.
It follows from Lemma \ref{espacetotal} that $F \modu y$ is local. Thus, no generality is lost in assuming that $F \modu y=0$.
This yields a linear form $\varphi$ on $\calH_4$ such that
$$F : M \mapsto \begin{bmatrix}
\varphi(M) \\
\varphi(M) \\
\varphi(M)
\end{bmatrix}.$$
Then, $\varphi$ is a linear function of the first row of matrices of $G$,
and the same holds for the second and third rows.
Obviously, the only possibility is that there is a triple $(\lambda,\mu,\nu) \in \K^3$ such that
$$\varphi : \begin{bmatrix}
a & b+c & f & h \\
b & d & a+c & i \\
c & e & g & a+b
\end{bmatrix} \mapsto \lambda\,a+\mu\,b+\nu\,c.$$
Applying Lemma \ref{ligneparligne} to the first row, we find $\mu=\nu$. Similarly, we obtain $\lambda=\nu$ by applying it to the second row,
and we conclude that $F=\lambda G$, where
$$G : \begin{bmatrix}
a & b+c & f & h \\
b & d & a+c & i \\
c & e & g & a+b
\end{bmatrix}\mapsto
\begin{bmatrix}
a+b+c \\
a+b+c \\
a+b+c
\end{bmatrix}.$$
If $G$ were local, looking at the first row would yield that $G$ maps every matrix of $\calH_4$ to the sum of
its first two columns, which is obviously false.

We finish by proving that $G$ is range-compatible.
Let $$M=\begin{bmatrix}
a & b+c & f & h \\
b & d & a+c & i \\
c & e & g & a+b
\end{bmatrix} \in \calH_4 \; \text{be such that $G(M) \neq 0$.}$$
If $a=b=c=1$, then
$G(M)$ is the first column of $M$. Assume now that $(a,b,c) \neq (1,1,1)$. Since $G(M) \neq 0$, we deduce that
exactly one of the scalars $a,b,c$ is non-zero. From the symmetry of the situation, we see that no generality
is lost in assuming that $a=1$ and $b=c=0$. In that case, if
$g=0$, then the $3 \times 3$ matrix obtained by deleting the second column of $M$ is seen to be invertible, whence $G(M) \in \im M$. If $g=1$, then $G(M)=(f+1)\,C_1(M)+C_3(M)$,
where $C_j(M)$ denotes the $j$-th column of $M$ for all $j \in \lcro 1,4\rcro$.
In any case, we have seen that $G(M) \in \im M$.

Therefore,
$$\calL_\rc(\calH_4)=\calL_\loc(\calH_4) \oplus \K G.$$
Using the Splitting Lemma, the case of spaces of Type 7 in Theorem \ref{specialRCF2} ensues.

\subsection{On the equivalence between spaces of special type}\label{inequivalentspecialtype}

Let $\calS$ be a linear subspace of $\calL(U,V)$.
We note that the matrix spaces representing $\calS^\bot$ are pairwise equivalent and hence their reduced subspaces are pairwise equivalent.
For each special type, we give such reduced subspaces:

\begin{center}
\begin{tabular}{| c | c |}
\hline
Type of $\calS$ & Reduced matrix subspace associated with $\calS^\bot$ \\
\hline
\hline
$1$, with $n$ rows & $\Mata_2(\F_2) \coprod \Mat_{2,n-2}(\F_2)$ \\
\hline
$2$ & $\Mata_3(\F_2)$ \\
\hline
$3$, with $n$ rows & $\Bigl\{\begin{bmatrix}
0 & a & b \\
a & b & c
\end{bmatrix} \mid (a,b,c)\in (\F_2)^3\Bigr\} \coprod \Mat_{2,n-3}(\F_2)$ \\
\hline
$4$ & $\calG_3^\bot=\Biggl\{
\begin{bmatrix}
0 & a & b+c \\
b & b & 0 \\
c & 0 & 0
\end{bmatrix} \mid (a,b,c)\in (\F_2)^3\Biggr\}$ \\
\hline
$5$ & $\calH_3^\bot=
\Biggl\{
\begin{bmatrix}
0 & a+b & c \\
b & a & 0 \\
a+c &  0  & a
\end{bmatrix} \mid (a,b,c)\in (\F_2)^3 \Biggr\}$
 \\
\hline
$6$ & $\calI_3^\bot=\Biggl\{
\begin{bmatrix}
a+c & 0 & b  \\
0 & b+c & a \\
c & c & c \\
\end{bmatrix} \mid (a,b,c)\in (\F_2)^3 \Biggr\}$ \\
\hline
$7$ & $\calH_4^\bot=\Biggl\{
\begin{bmatrix}
b+c & a+c & a+b \\
a & 0 & 0 \\
0 & b & 0 \\
0 & 0 & c
\end{bmatrix} \mid (a,b,c)\in (\F_2)^3\Biggr\}$ \\
\hline
\end{tabular}
\end{center}

Note that $\calH_3^\bot$ is equivalent to the space denoted by $\calU_3(\F_2)$ in Proposition
\ref{primitivetypeprop1}, which one obtains by performing the column and row operations
$C_2 \leftrightarrow C_3$, $C_3 \leftrightarrow C_1$ and $L_1 \leftrightarrow L_3$.

From there, we can prove the following result:

\begin{prop}
Let $\calS$ be a linear subspace of $\calL(U,V)$. Then, $\calS$ is of at most one of Types 1 to 7.
\end{prop}

\begin{proof}
Two matrix spaces which represent the reduced space associated with $\calS^\bot$ must be equivalent.
Therefore, in order to prove the claimed result, it suffices to show that the spaces listed in the above
array are pairwise inequivalent. By considering the number of rows, we deduce that a space of Type 7 can be of none of Types
1 to 6, and that a space of Type 2, 4, 5 or 6 can be of none of Types 1 and 3.
If $\calS$ has Type 1, then we see that the set $\bigl\{y \in V : \dim \calS^\bot y<2\bigr\}$
is a $2$-dimensional linear subspace of $V$, whereas if $\calS$ has Type 3 one checks that
this space has dimension $1$ (for the special case given in the above array, this space is spanned by the first vector of the canonical basis).
Thus, a space of Type 1 cannot be of Type 3.

To conclude the proof, we need to differentiate between spaces of Types 2, 4, 5 and 6.
Noting that $\calG_3^\bot$ is equivalent to a subspace of the space $\calJ_3(\F_2)$ from Proposition \ref{primitivetypeprop3},
we deduce from Proposition \ref{primitivetypeprop3} that a space can have at most one of Types 2, 4 or 5.

Finally, using Propositions \ref{primitivetypeprop1} and \ref{primitivetypeprop2}, we know that, in each space $\Mata_3(\F_2)$, $\calG_3^\bot$ and $\calH_3^\bot$, every matrix has rank at most $2$. However, the space $\calI_3^\bot$ contains the rank $3$ matrix
$\begin{bmatrix}
1 & 0 & 0  \\
0 & 1 & 0 \\
1 & 1 & 1
\end{bmatrix}$. Therefore, a space of Type 6 can have neither Types 2, 4 nor 5.
This concludes the proof.
\end{proof}

\subsection{An additional property of spaces of special type}

The following result will be used later in our proof of Theorem \ref{classRCF2}.

\begin{prop}\label{specialtypelemma}
Let $\calS$ be one of the spaces $\Mata_3(\F_2)$, $\calG_3^\bot$, $\calH_3^\bot$, $\calI_3^\bot$ or $\calH_4^\bot$.
Then, $\dim \calS x\geq 2$ for all $x \in (\F_2)^3 \setminus \{0\}$,
unless $\calS$ equals $\calG_3^\bot$ in which case exactly one vector $x \in (\F_2)^3 \setminus \{0\}$
satisfies $\dim \calS x\leq 1$. \\
Moreover,
$$\dim\bigl(\Vect \{Ny \mid N \in \calS, \; y \in \K^3\}\bigr) \geq 3.$$
\end{prop}

To prove this result, we start with an interesting observation, which is obtained by straightforward computations:

\begin{lemma}\label{selfduality}
Let $\calS$ be one of the spaces $\Mata_3(\F_2)$, $\calG_3^\bot$, $\calH_3^\bot$ or $\calH_4^\bot$.
Then, $\widehat{\calS}$ is equivalent to $\calS$. \\
On the other hand, if $\calS=\calI_3^\bot$ then $\widehat{\calS}$ is represented by the matrix space
$$\Biggl\{\begin{bmatrix}
a & b & a \\
b & c & c \\
0 & 0 & a+b+c
\end{bmatrix} \mid (a,b,c)\in (\F_2)^3\Biggr\}.$$
\end{lemma}

\begin{proof}[Proof of Proposition \ref{specialtypelemma}]
The second statement is obvious.
For the first one, we use Lemma \ref{selfduality}.
As $\calH_3^\bot$ is equivalent to the space $\calU_3(\F_2)$ from Proposition \ref{primitivetypeprop1},
every non-zero matrix that belongs to $\Mata_3(\F_2)$ or $\calH_3^\bot$ has rank $2$.
Moreover, it is easily seen that $\calH_4^\bot$ contains no rank $1$ matrix, for if the matrix
$$M=\begin{bmatrix}
b+c & a+c & a+b \\
a & 0 & 0 \\
0 & b & 0 \\
0 & 0 & c
\end{bmatrix}$$
has rank $1$ then exactly one of $a,b,c$ equals $1$ (as seen from the last three rows),
and then one sees that $M$ has two linearly independent columns.
Finally, $\begin{bmatrix}
0 & 1 & 0 \\
0 & 0 & 0 \\
0 & 0 & 0
\end{bmatrix}$ is the sole rank $1$ matrix of $\calG_3^\bot$ (indeed, a rank $1$ matrix of the form
$\begin{bmatrix}
0 & a & b+c \\
b & b & 0 \\
c & 0 & 0
\end{bmatrix}$ should have at most one non-zero entry among $b$, $c$ and $b+c$, and hence $b=c=0$).

Thus, by Lemma \ref{selfduality} we find that $\calS$ satisfies the first statement provided that
it is not equivalent to $\calI_3^\bot$.

Finally, if a matrix $\begin{bmatrix}
a & b & a \\
b & c & c \\
0 & 0 & a+b+c
\end{bmatrix}$ has rank $1$, then $(a,b,c) \neq (0,0,0)$ and hence $\begin{bmatrix}
a & b \\
b & c
\end{bmatrix}$ has rank $1$ and $a+b+c=0$. As $\begin{bmatrix}
1 & 0 \\
0 & 0
\end{bmatrix}$, $\begin{bmatrix}
0 & 0 \\
0 & 1
\end{bmatrix}$ and $\begin{bmatrix}
1 & 1 \\
1 & 1
\end{bmatrix}$ are the sole rank $1$ matrices in $\Mats_2(\F_2)$, we conclude that no $3 \times 3$ rank $1$ matrix of the above form exists.
Thus, by using Lemma \ref{selfduality} we conclude that $\dim (\calI_3^\bot y) \geq 2$ for all $y \in (\F_2)^3 \setminus \{0\}$.
\end{proof}

\section{Proof of the main theorem (1)}\label{mainproofsec1}

In this section, we prove Theorem \ref{classRCF2} in the special case when $n = 2$,
and we show that if it holds for $n=3$ then it also holds for all greater values of $n$.

Throughout the section, we set $\K=\F_2$.

\subsection{The case $n=2$}

Here, we assume that $n=2$. Let $\calS$ be a linear subspace of $\Mat_{n,p}(\K)$ with codimension $2n-3$.
Then, $\calS^\bot$ contains exactly one non-zero matrix $B$, and hence:
\begin{itemize}
\item Either $B$ has rank $2$, whence it is equivalent to $\begin{bmatrix}
I_2 \\
[0]_{(p-2) \times 2}
\end{bmatrix}$, which shows that $\calS$ is equivalent to $\Mats_2(\K) \coprod \Mat_{2,p-2}(\K)$, i.e.
$\calS$ has Type 1;
\item Or $B$ has rank $1$, whence $\calS$ is equivalent to $\K \vee \Mat_{1,p-1}(\K)$, and one deduces
from Lemma \ref{vecteurdim1} that every range-compatible linear map on $\calS$ is local.
\end{itemize}

\subsection{General considerations}

In the rest of this section and in the next one, we assume that
$n > 2$. We also assume that Theorem \ref{classRCF2} holds for all matrix spaces with $n-1$ rows.
Let $\calS$ be a linear subspace of $\Mat_{n,p}(\K)$, interpreted as a space of linear maps from $\K^p$ to $\K^n$, and let
$$F : \calS \rightarrow \K^n$$
be a non-local range-compatible linear map.
To simplify the notation, we set
$$U:=\K^p \quad \text{and} \quad V:=\K^n.$$

As a consequence of Lemma \ref{vecteurdim1}, the assumption that $F$ is non-local yields:

\begin{claim}\label{notrans1claim}
There is no non-zero vector $x \in U$ such that $\dim \calS x\leq 1$.
\end{claim}

Throughout the proof, it will be necessary to discuss the nature of $\calS$-adapted vectors.
As $\dim \calS =2n-3$, a vector $y \in V$ is $\calS$-adapted if and only if
$$\dim \calS^\bot y \geq 2.$$
We qualify such vectors according to the nature of the map $F \modu y$
that they induce:

\begin{Def}
Let $y$ be an $\calS$-adapted vector of $V$.
We say that $y$ has \textbf{Type $0$ for $F$} whenever $F \modu y$ is local. \\
Given $i \in \lcro 1,7\rcro$, we say that $y$ has \textbf{Type $i$ for $F$} whenever the space
$\calS \modu y$ has Type $i$ and $F \modu y$ is non-local.
\end{Def}

In particular, since $F$ is non-local the following result is a consequence of Lemma \ref{3vectorslemma}:

\begin{claim}\label{3vectorsclaim}
The set of all $\calS$-adapted vectors of Type $0$ is included in a $2$-dimensional subspace of $V$.
\end{claim}

In addition, we need to introduce another special type of vectors:

\begin{Def}
A vector $z$ of $V$ is called \textbf{super-$\calS$-adapted} when $\dim \calS^\bot z \geq 3$.
\end{Def}

Given a super-$\calS$-adapted vector $z$, we see by duality that
$\dim(\calS \modu z)<2(n-1)-3$, and hence Theorem \ref{maintheolin} yields:

\begin{claim}
For every super-$\calS$-adapted vector $z \in V$, the map $F \modu z$ is local.
\end{claim}

\subsection{The case when $n\geq 4$}

In this section, we prove the inductive step in the case when $n \geq 4$.
We start by discarding most types of $\calS$-adapted vectors:

\begin{claim}
Every $\calS$-adapted vector has Type 0, 1 or 3 for $F$.
\end{claim}

\begin{proof}
Assume on the contrary that some $\calS$-adapted vector $y_0$ has Type 2 or one of Types 4 to 7 for $F$.
Then, $n=4$. We shall prove that there exists a linearly independent triple of
super-$\calS$-adapted vectors, which will contradict Claim \ref{3vectorsclaim}.

Without loss of generality, we may assume that $y_0$ is the last vector of the canonical basis $(y_1,y_2,y_3,y_4)$ of
$V$. Then, no further generality is lost in assuming that every matrix $M$ of $\calS$ splits up as
$$M=\begin{bmatrix}
K(M) \\
[?]_{1 \times p}
\end{bmatrix}$$
and either
$$K(\calS)=\calV \coprod \Mat_{3,p-3}(\F_2),$$
where $\calV$ is one of the spaces $\Mats_3(\F_2)$, $\calG_3$, $\calH_3$ or $\calI_3$,
or
$$K(\calS)=\calH_4 \coprod \Mat_{3,p-4}(\F_2).$$
We denote by $\calT$ the space of all matrices $N$ of $\calS^\bot$ such that $Ny_4=0$.
Then, either $F \modu y_4$ has Type $2$ or one of Types $4$ to $6$, and hence
$\calT$ is the space of all matrices of the form
$$\begin{bmatrix}
A & [0]_{3 \times 1}\\
[0]_{(p-3) \times 3} & [0]_{(p-3) \times 1}
\end{bmatrix} \quad \text{with $A \in \calV^\bot$,}$$
or $\calS \modu y_4$ has Type $7$ and $\calT$ is the space of all matrices of the form
$$\begin{bmatrix}
A & [0]_{4 \times 1}\\
[0]_{(p-4) \times 3} & [0]_{(p-4) \times 1}
\end{bmatrix} \quad \text{with $A \in \calH_4^\bot$.}$$
As $y_4$ is $\calS$-adapted and $\codim (\calS \modu y_4)=3=\codim \calS-2$, the space
$$P:=\calS^\bot y_4$$
has dimension $2$.
Let $y \in \Vect(y_1,y_2,y_3) \setminus \{0\}$ be such that $\dim \calT y \geq 2$.
We claim that one of the following conditions must hold:
\begin{enumerate}[(i)]
\item $y$ is super-$\calS$-adapted;
\item $y+y_4$ is super-$\calS$-adapted;
\item $\calT y=P$.
\end{enumerate}
Assume that none of $y$ and $y+y_4$ is super-$\calS$-adapted.
Then, $\dim \calS^\bot y \leq 2$ and $\dim \calS^\bot (y+y_4) \leq 2$.
However, as $\calT y \subset \calS^\bot y$ and $\dim \calT y \geq 2$, we find $\calT y=\calS^\bot y$.
By the very definition of $\calT$, we see that $\calT(y+y_4)=\calT y$, and hence $\calT y=\calT (y+y_4)=\calS^\bot (y+y_4)$ with the above line of reasoning. In particular, for all $N \in \calS^\bot$, we have $N y_4=N(y+y_4)-Ny  \in \calT y$, whence $P\subset \calT y$. As the dimensions are equal on both sides, we conclude that condition (iii) holds.

Now, we can conclude. By Proposition \ref{specialtypelemma}, the space
$\Vect \{N y \mid y \in \Vect(y_1,y_2,y_3), \; N \in \calT\}$ has dimension greater than $2$,
and hence the set of all vectors $y \in \Vect(y_1,y_2,y_3) \setminus \{0\}$ for which $\calT y=P$
is included in a hyperplane $H$ of $\Vect(y_1,y_2,y_3)$.
On the other hand, denoting by $D$ the set of all non-zero vectors $z \in \Vect(y_1,y_2,y_3) \setminus \{0\}$ for which $\dim \calT z\leq 1$,
we know from Proposition \ref{specialtypelemma} that $D$ has at most one element; by Lemma \ref{coveringlemma}, the
set $\Vect(y_1,y_2,y_3) \setminus (D \cup H)$ is not included in a hyperplane of $\Vect(y_1,y_2,y_3)$.
It follows that we can extract a basis $(z_1,z_2,z_3)$ of $\Vect(y_1,y_2,y_3)$ from this set,
to the effect that, for each $i \in \{1,2,3\}$, we can find a scalar $t_i \in \K$ such that $z_i+t_i y_4$ is
super-$\calS$-adapted. Then, $(z_i+t_i\,y_4)_{1 \leq i \leq 3}$ is obviously a linearly independent triple
of super-$\calS$-adapted vectors, contradicting Claim \ref{3vectorsclaim}.
Therefore, every $\calS$-adapted vector has Type 0, 1 or 3 for $F$.
\end{proof}

In the next step, we reduce the situation to the case $p=2$.

\begin{claim}
There is at least one $\calS$-adapted vector of Type 1 or 3 for $F$, and
there exists a $2$-dimensional subspace $P$ of $U$ which contains the range of every matrix of $\calS^\bot$.
\end{claim}

\begin{proof}
Given a vector $y$ of Type 1 or 3 for $F$, we denote by
$\calS^\bot_y$ the space of all matrices $N \in \calS^\bot$ for which $N y=0$.
Then, it is obvious from the definition of spaces of Types 1 and 3 that
$\calS^\bot_y$ has dimension $2n-5$ and that there is a unique
$2$-dimensional subspace $P_y$ of $U$ that contains the image of every matrix of $\calS^\bot_y$.
On top of that, assume that we have a vector $z$ of $V \setminus \K y$ such that:
\begin{enumerate}[(i)]
\item $z$ is $\calS$-adapted of Type 1 or 3;
\item $\overline{z}$ is $(\calS \modu y)$-adapted, where $\overline{z}$ denotes the class of $z$ modulo $\K y$;
\item There is a rank $2$ matrix $N$ in $\calS^\bot$ such that $Ny=Nz=0$.
\end{enumerate}

As $\calS \modu y$ has Type 1 or 3, it is obvious that
$\dim (\calS \modu y)^\bot y' \leq 2$ for every non-zero vector $y' \in V/\K y$.
Thus, as $\overline{z}$ is $(\calS \modu y)$-adapted we have
$\dim (\calS \modu y)^\bot \overline{z}=2$,
whence the space of all matrices $N \in \calS^\bot_y$ for which $N z=0$ has codimension $2$ in $\calS^\bot_y$.
In other words, $\dim(\calS^\bot_y \cap \calS^\bot_z)=2(n-2)-3$,
whence $\dim(\calS^\bot_y +\calS^\bot_z)=2\bigl(2(n-1)-3\bigr)-\bigl(2(n-2)-3\bigr)=2n-3$
and it follows that $\calS^\bot_y +\calS^\bot_z=\calS^\bot$.
As assumption (iii) yields $P_y=P_z$, we deduce from $\calS^\bot_y +\calS^\bot_z=\calS^\bot$
that $P_y$ contains the range of every matrix of $\calS^\bot$.

In the rest of the proof, we demonstrate the existence of a pair $(y,z)$ satisfying conditions (i) to (iii) above, which will complete
the proof.

We know from Claim \ref{3vectorsclaim} that the set of all $\calS$-adapted vectors of Type 0 for $F$ is included in a $2$-dimensional subspace $G_1$ of $V$.
Moreover, as $F$ is non-local, Lemma \ref{adaptedvectorslemma} shows that
the set of all non-$\calS$-adapted vectors is included in a hyperplane $H$ of $V$.
By Lemma \ref{coveringlemma}, $H \cup G_1$ is a proper subset of $V$, which shows that some
$\calS$-adapted vector $y$ has Type 1 or 3 for $F$.
Without loss of generality, we may assume that $y$ is
the last vector of the canonical basis $(y_1,\dots,y_n)$ of $V$.

Therefore, by further reducing the situation, we see that no generality is lost in assuming that
every matrix of $\calS$ splits up as
$$M=\begin{bmatrix}
K(M) \\
[0]_{(n-4) \times p} \\
[?]_{1 \times p}
\end{bmatrix}$$
and that
$$K(\calS)=\Mats_2(\K) \vee \Mat_{1,p-2}(\K) \quad \text{or} \quad K(\calS)=\calV_2 \coprod \Mat_{3,p-2}(\K).$$

Given $y' \in V$, we shall denote by $\overline{y'}$ its class in $V/\K y_n$.
Set $W_1:=\Vect(y_1,y_2,y_n)$.
As $W_1$ is included in a hyperplane of $\K^n$, Lemma \ref{coveringlemma} yields
a vector $z$ in $V \setminus (H \cup G_1 \cup W_1)$, to the effect that
$z \not\in W_1$ and $z$ is an $\calS$-adapted vector of Type 1 or 3 for $F$.

As $z \not\in W_1$, we see that $\overline{z} \not\in \Vect(\overline{y_1},\overline{y_2})$.
Then, we note that $\calS \modu \Vect(y,z)$ has Type 1 or Type 3. Indeed:
\begin{itemize}
\item Either $\calS \modu y$ is represented by $\Mats_2(\K) \vee \Mat_{n-3,p-2}(\K)$, and hence
it is obvious that $\calS \modu \Vect(y,z)$ has Type 1.
\item Or $\calS \modu y$ is represented by $\calV_2 \vee \Mat_{n-4,p-2}(\K)$;
if in addition $z \not\in \Vect(y_1,y_2,y_3)$ then $\calS \modu \Vect(y,z)$ has Type 3;
otherwise $z=\lambda y_1+\mu y_2+y_3$ for some $(\lambda,\mu)\in \K^2$, and
then $\calS \modu \Vect(y,z)$ is represented by the matrix space
$\calW \vee \Mat_{n-4,p-2}(\K)$, where
$$\calW=\Biggl\{\begin{bmatrix}
a+\lambda c & b \\
b+\mu c & c
\end{bmatrix} \mid (a,b,c)\in \K^3\Biggr\}.$$
Then, using the column operation $C_1 \leftarrow C_1+\mu C_2$ yields that $\calW$ is equivalent to $\Mats_2(\K)$,
and we deduce that $\calS \modu \Vect(y,z)$ has Type 1.
\end{itemize}
It follows that $\overline{z}$ is $(\calS \modu y)$-adapted and that
$\calS^\bot$ contains a rank $2$ matrix which vanishes at $z$ and $y$
(as this is equivalent to the existence of a rank $2$ operator in $(\calS \modu \Vect(y,z))^\bot$).
Thus, the pair $(y,z)$ satisfies conditions (i) to (iii) above, which completes our proof.
\end{proof}

From there, no generality is lost in assuming that the range of every matrix of $\calS^\bot$
is included in $\K^2 \times \{0\}$, to the effect that $\calS$ splits up as $\calT \coprod \Mat_{n,p-2}(\K)$
for some $3$-dimensional subspace $\calT$ of $\Mat_{n,2}(\K)$, and
$F$ splits up as $F=G \coprod H$, where $G$ and $H$ are range-compatible linear maps, respectively, on $\calT$
and $\Mat_{n,p-2}(\K)$. As $H$ is local by Theorem \ref{espacetotal}, the map $G$ is non-local.
If we demonstrate that $\calT$ has Type 1 or 3, then we will obtain that $\calS$ has Type 1 or 3, and the proof will be complete.

Thus, from now on we can assume that $p=2$.
Consider the space
$$\calS':=\Bigl\{\begin{bmatrix}
M & F(M)
\end{bmatrix} \mid M \in \calS\Bigr\} \subset \Mat_{n,3}(\K).$$
As $F$ is range-compatible, every matrix in $\calS'$ has rank less than $3$.
We shall complete the proof by using some results from the classification of matrix spaces with upper-rank at most $2$.

Note first that no non-zero vector belongs to the kernel of every matrix of $\calS'$.
Indeed, if such a vector $x$ existed, then $x \in \K^2 \times \{0\}$ otherwise $F$ would be local,
and then we would find $\calS x=\{0\}$, contradicting Claim \ref{notrans1claim}.
Thus, $\calS'$ satisfies condition (i) in the definition of a reduced space.
Without loss of generality, we can assume that the sum of the ranges of the matrices in $\calS'$
equals $\K^m \times \{0\}$ for some $m \in \lcro 2,n\rcro$ (we must have $m \geq 2$ because of Claim \ref{notrans1claim}).

Assume that $n=m$, to the effect that $\calS'$ is reduced.
As $n \geq 2$, we see that $\calS'$ cannot have upper rank $1$, whence $\calS'$ has upper rank $2$.
As $n \geq 4$ and $\calS'$ is reduced, Proposition \ref{primitivedimprop} shows that $\calS'$
cannot be primitive. As $\calS'$ is reduced and $m \geq 3$, we deduce from Proposition \ref{nonprimitiveprop}
that $\calS'$ must be equivalent to a linear subspace of the space of all matrices of the form
$$\begin{bmatrix}
? & [?]_{1 \times 2} \\
[?]_{(n-1) \times 1} & [0]_{(n-1) \times 2}
\end{bmatrix},$$
yielding a $2$-dimensional subspace $P$ of $\K^3$ such that
$$\forall x \in P, \; \dim \calS' x \leq 1.$$
However, we would then find a non-zero vector $x \in P \cap (\K^2 \times \{0\})$, yielding a non-zero vector $x' \in \K^2$
such that $\dim \calS x' \leq 1$, in contradiction with Claim \ref{notrans1claim}.

We deduce that $m<n$. Then, we write every matrix $M$ of $\calS$ as
$M=\begin{bmatrix}
H(M) \\
[0]_{(n-m) \times 2}
\end{bmatrix}$, with $H(M) \in \Mat_{m,2}(\K)$, and we recover a non-local range-compatible linear map $f : H(\calS) \rightarrow \K^m$ such that
$$F : M \mapsto \begin{bmatrix}
f(H(M)) \\
[0]_{(n-m) \times 1}
\end{bmatrix}.$$
Then, by induction, we know that $H(\calS)$ must be of Type 1 or 3, whence $\calS$ has Type 1 or 3.

This completes the proof for $n>3$, assuming that Theorem \ref{classRCF2} holds in the case $n=3$.

\section{Proof of the main theorem (2): The case $n=3$}\label{mainproofsec2}

Our aim in this section is to prove Theorem \ref{classRCF2} in the special case $n=3$.
By the results of the preceding section, we know that doing so will complete the proof of Theorem \ref{classRCF2}.
Throughout the section, we set $\K:=\F_2$.
Let $\calS$ be a linear subspace of $\Mat_{3,p}(\K)$ with codimension $3$, and assume that there is a non-local range-compatible
linear map
$$F : \calS \rightarrow \K^3.$$
Our goal is to prove that $\calS$ has one of Types 1 to 7.
We shall do this by slowly gathering information on the structure of the $\calS^\bot$ space
and of its dual space $\widehat{\calS^\bot}$.
Remember the notation
$$U=\K^p \quad \text{and} \quad V=\K^3.$$
Note that Claims \ref{notrans1claim} and \ref{3vectorsclaim} hold.
As $n=3$, remark also that any $\calS$-adapted vector must have one of Types 0 or 1 for $F$.

Let us quickly explain the structure of the proof.
In Section \ref{n=3prelim}, we gather some general results on $\calS^\bot$.
In Section \ref{n=3rank1}, we obtain information on the possible rank $1$ matrices of $\calS^\bot$.
Afterwards, we shall split the discussion into four cases (Sections \ref{n=3case1}, \ref{n=3case2}, \ref{n=3case3} and \ref{n=3case4}),
whether $\calS^\bot$ contains rank $1$ matrices or not, and whether there is a super-$\calS$-adapted vector or not.

\subsection{Preliminary results}\label{n=3prelim}

We start by stating obvious corollaries of Claim \ref{notrans1claim}, Lemma \ref{adaptedvectorslemma} and Lemma \ref{vecteurdim1}:

\begin{claim}\label{notworank1samerange}
Distinct rank $1$ matrices of $\calS^\bot$ have distinct ranges.
\end{claim}

\begin{claim}
The set of all non-$\calS$-adapted vectors is included in a $2$-dimensional subspace of $V$.
\end{claim}

Next, we investigate the possible dimensions of $\calS^\bot y$ for $y \in V$.

\begin{claim}\label{no0trans}
There is no vector $y \in V \setminus \{0\}$ such that $\calS^\bot y=\{0\}$.
\end{claim}

\begin{proof}
Assume that there is such a vector $y$.
Let $z$ be an $\calS$-adapted vector. Then, $z \not\in \K y$.
We contend that $F \modu z$ is local.
Indeed, if this were not the case, then the induction hypothesis would yield that $\calS \modu z$
is represented by $\Mats_2(\K) \coprod \Mat_{2,p-2}(\K)$ in some bases of $U$ and $V/\K z$,
yielding some $A \in \calS^\bot$ such that $\Ker A=\K z$.
Then, $Ay \neq 0$, contradicting our assumptions.

We know that some $2$-dimensional subspace $P$ of $V$ contains every non-$\calS$-adapted vector.
By Lemma \ref{coveringlemma}, the set $V \setminus P$
is not included in a hyperplane of $V$, whence we may find three linearly independent vectors
of $V$ outside of $P$. Then, by Lemma \ref{3vectorslemma}, $F$ is local. This is a contradiction.
\end{proof}

As a consequence, we obtain:

\begin{claim}\label{notsuperclaim}
For every non-zero vector $y \in V$ that is not $\calS$-adapted,
the space $\calS$ contains a $(p-1)$-dimensional subspace in which all the matrices have their image included in $\K y$.
\end{claim}

Now, we examine the super-$\calS$-adapted vectors more closely.

\begin{claim}
Let $y_1$ and $y_2$ be distinct $\calS$-adapted vectors, with
$y_1$ super-$\calS$-adapted. Then, $y_2$ has Type 1 for $F$.
\end{claim}

\begin{proof}
Assume on the contrary that $y_2$ does not have Type 1 for $F$. Then, by induction
the map $F \modu y_2$ is local.

Note that $y_1$ and $y_2$ are linearly independent since the underlying field is $\F_2$.
As $y_1$ is super-$\calS$-adapted, we have $\codim (\calS \modu y_1) \leq 3-3=0$, whence
$\calS \modu y_1=\calL(\calS,V/\K y_1)$.
In particular, $F \modu y_1$ is local.
As $F \modu y_2$ is local, we can subtract a local map from $F$
to reduce the situation to the one where $F \modu y_2=0$.
Then, we have a vector $x \in U$ such that
$$\forall s \in \calS, \quad F(s)=s(x) \modu \K y_1 \quad \text{and} \quad F(s) \in \K y_2.$$
In particular, this yields $s(x) \in \Vect(y_1,y_2)$ for all $s \in \calS$.
If $x=0$ then $F=0$ as $\K y_1 \cap \K y_2=\{0\}$, contradicting the fact that $F$ is non-local.
Thus, $x \neq 0$. Then, as $n=3$ and $\calS \modu y_1=\calL(\calS,V/\K y_1)$, we can choose $s \in \calS$ such that
$s(x)\not\in \Vect(y_1,y_2)$, contradicting the above result.

Therefore, $y_2$ has Type 1 for $F$.
\end{proof}

As a super-$\calS$-adapted vector is always of Type 0 for $F$, we deduce:

\begin{claim}\label{atmost1superadapted}
There is at most one super-$\calS$-adapted vector.
\end{claim}

We finish with a counting result that will be used in several instances:

\begin{claim}\label{countingclaim}
For every positive integer $i$, denote by $m_i$ the number of rank $i$ matrices in $\calS^\bot$, and by
$n_i$ the number of vectors $y \in V$ for which $\dim \calS^\bot y=i$. Then,
$$3m_1+m_2=3n_1+n_2.$$
\end{claim}

\begin{proof}
We count the set $\calN:=\{(N,y)\in (\calS^\bot \setminus \{0\}) \times (V \setminus \{0\}) : Ny=0\}$
in two different ways. For each $y \in \calS^\bot \setminus \{0\}$, the linear map
$\widehat{y}: N\in \calS^\bot \mapsto Ny$ has exactly $2^{3-\dim \calS^\bot y}-1$ non-zero vectors in its kernel,
that is, there are as many matrices $N \in \calS^\bot$ for which $(N,y)\in \calN$.
Therefore, $\# \calN=3n_1+n_2$.
On the other hand, for each $N \in \calS^\bot$, there are $2^{3-\rk N}-1$ elements $y$ of $V \setminus \{0\}$
such that $(N,y) \in \calN$. Thus, $\# \calN=3m_1+m_2$, and the claimed result ensues.
\end{proof}

\subsection{General results on the rank $1$ matrices in $\calS^\bot$}\label{n=3rank1}

Here, we consider the existence of rank $1$ matrices in $\calS^\bot$
and we gather additional information on the situation where we can find one or several adapted vectors
in the kernel of such a matrix.

\begin{claim}\label{specrank1claim}
Let $A$ be a rank $1$ matrix of $\calS^\bot$. Let $y$ be an $\calS$-adapted vector in $\Ker A$.
Then, $F \modu y$ is local and $y$ is not super-$\calS$-adapted. Moreover, if there is a $1$-dimensional subspace $D$ of $V/\K y$
such that $(F \modu y)(\overline{s}) \in D$ for all $\overline{s} \in \calS \modu y$, and $D \neq \Ker A/\K y$,
then $F \modu y=0$.
\end{claim}

\begin{proof}
Set $x \in \im A \setminus \{0\}$.
By Claim \ref{notrans1claim}, we have $\dim \calS x\geq 2$, and obviously $\calS x \subset \Ker A$, whence $\calS x=\Ker A$.
As $y \in \Ker A$, it follows that $(\calS \modu y) x=\Ker A/ \K y$ has dimension $1$,
whence $F \modu y$ cannot have Type 1 and $y$ is not super-$\calS$-adapted.
By induction, $F \modu y$ is local, which yields $x' \in U$ such that
$$\forall s \in \calS, \; F(s)=s(x')\modu \K y.$$
We have seen that $\codim (\calS \modu y)=1$ and $\calS \modu y$ is included in the space $\calT$
of all linear maps $s : U \rightarrow V/\K y$ for which $s(x) \in \Ker A/\K y$, which also has codimension $1$ in
$\calL(U,V/\K y)$. Therefore, $\calS \modu y=\calT$.

Assume now that $x' \neq 0$ and let $D$ be a $1$-dimensional subspace of $V/\K y$
such that $(F \modu y)(\overline{s}) \in D$ for all $\overline{s} \in \calS \modu y$.
If $x' \neq x$, then we use $\calS \modu y=\calT$ to find that $\{\overline{s}(x') \mid \overline{s} \in \calS \modu y\}$ has dimension $2$,
which contradicts our assumption on $D$. Therefore, $x'=x$; choosing $\overline{s} \in \calS \modu y$ such that $\overline{s}(x) \neq 0$,
we deduce that $D=\Ker A/\K y$, which concludes the proof.
\end{proof}

\begin{claim}\label{adaptedoutsidekernel}
Let $A$ be a rank $1$ matrix of $\calS^\bot$.
Let $z$ be an $\calS$-adapted vector of $V \setminus \Ker A$.
Assume that some $y \in \Ker A$ is $\calS$-adapted.
Then, $F \modu z$ is non-local.
\end{claim}

\begin{proof}
Assume on the contrary that $F \modu z$ is local.
Then, we lose no generality in assuming that $F \modu z=0$,
whence $F(s) \in \K z$ for all $s \in \calS$.
Denoting by $\overline{z}$ the class of $z$ in $V/\K y$,
we deduce that $(F \modu y)(\overline{s}) \in \K \overline{z}$ for all $\overline{s} \in \calS \modu y$.
However, it is obvious that $\K \overline{z} \neq \Ker A/\K y$, whence
Claim \ref{specrank1claim} yields $F \modu y=0$. As $\K y \cap \K z=\{0\}$,
we recover $F=0$ from $F \modu y=0$ and $F \modu z=0$, contradicting our assumption that $F$ be non-local.
\end{proof}

\begin{claim}\label{intersectionkernelclaim}
Let $A$ and $B$ be distinct rank $1$ matrices of $\calS^\bot$.
Let $y \in (\Ker A \cap \Ker B) \setminus \{0\}$.
Then, $F \modu y$ is local and $y$ is non-$\calS$-adapted.
\end{claim}

\begin{proof}
By Claim \ref{notworank1samerange}, the matrices $A$ and $B$ do not have the same image.
Set $x_1\in \im A \setminus \{0\}$ and $x_2 \in \im B \setminus \{0\}$.
Then, we see that
$(\calS \modu y)x_1 \in \Ker A/\K y$ and $(\calS \modu y)x_2 \in \Ker B/\K y$.
On the other hand, the space $\calT$ of all linear maps $u$ from $U$ to $V/\K y$ which
satisfy $u(x_1) \in \Ker A/\K y$ and $u(x_2) \in \Ker B/\K y$ has obviously codimension $2$ in $\calL(U,V/\K y)$,
and, as $\calS^\bot y \neq \{0\}$ by Claim \ref{no0trans}, we see that $\codim (\calS \modu y) \leq 2$, whence
$\calS \modu y=\calT$. Thus, in well-chosen bases, $\calS \modu y$ is represented by
$\calD_1 \coprod \calD_2 \coprod \Mat_{2,p-2}(\K)$, where each $\calD_i$ is a $1$-dimensional subspace of $\K^2$.
As each range-compatible linear map on $\calD_1$ (respectively $\calD_2$, respectively $\Mat_{2,p-2}(\K)$) is local,
we conclude that $F \modu y$ is local. As $\codim (\calS \modu y)=2$, we also see that $y$ is non-$\calS$-adapted.
\end{proof}

\begin{claim}\label{tworank1samekernel}
There do not exist rank $1$ matrices $A$ and $B$ in $\calS^\bot$ with distinct kernels.
\end{claim}

\begin{proof}
Assume one the contrary that such matrices $A$ and $B$ exist.
Claim \ref{notworank1samerange} yields $\im A \neq \im B$.
On the other hand, $D:=\Ker A \cap \Ker B$ has dimension $1$.
Define $y_1$ as the sole non-zero vector of $D$.
By Claim \ref{intersectionkernelclaim}, the map $F \modu y_1$ is local.

Moreover, $y_1$ is not $\calS$-adapted.
If we could find $\calS$-adapted vectors $y_2 \in \Ker A \setminus \{y_1\}$ and
$y_3 \in \Ker B \setminus \{y_1\}$, then Claim \ref{specrank1claim} would yield that $F \modu y_2$ and $F \modu y_3$ are local,
and obviously $(y_1,y_2,y_3)$ would be a basis of $V$; then Lemma \ref{3vectorslemma}
would yield that $F$ is local, contradicting our assumptions.

It follows that one of the planes $\Ker A$ or $\Ker B$ contains only non-$\calS$-adapted vectors.
Without loss of generality, we may assume that all the vectors of $\Ker A$ are non-$\calS$-adapted.
Replacing $\calS$ with an equivalent space, we may also assume that
$$A=\begin{bmatrix}
0 & 0 & 0 \\
0 & 0 & 1 \\
[0]_{(p-2) \times 1} & [0]_{(p-2) \times 1} & [0]_{(p-2) \times 1}
\end{bmatrix} \quad \text{and} \quad
B=\begin{bmatrix}
0 & 1 & 0 \\
0 & 0 & 0 \\
[0]_{(p-2) \times 1} & [0]_{(p-2) \times 1} & [0]_{(p-2) \times 1}
\end{bmatrix}.$$
Thus, $\Ker A=\K^2 \times \{0\}$ and every matrix of $\calS$ has the form
$$\begin{bmatrix}
? & ? & [?]_{1 \times (p-2)} \\
0 & ? & [?]_{1 \times (p-2)} \\
? & 0 & [?]_{1 \times (p-2)}
\end{bmatrix}.$$
Denote by $(f_1,f_2,f_3)$ the canonical basis of $\K^3$.
From there, every rank $1$ matrix of $\calS$ with image spanned by $f_1+f_2=\begin{bmatrix}
1 \\
1 \\
0
\end{bmatrix}$ must have its first column zero.
The vector $f_1+f_2$ is non-$\calS$-adapted as it belongs to $\Ker A$, whence
$\calS$ contains every matrix of the form
$\begin{bmatrix}
0 & L \\
0 & L \\
0 & [0]_{1 \times (p-1)}
\end{bmatrix}$ with $L \in \Mat_{1,p-1}(\K)$ (this uses Claim \ref{notsuperclaim}).
As $f_2$ is non-$\calS$-adapted, we also obtain that $\calS$ contains every matrix of the form
$\begin{bmatrix}
0 & [0]_{1 \times (p-1)} \\
0 & L \\
0 & [0]_{1 \times (p-1)}
\end{bmatrix}$ with $L \in \Mat_{1,p-1}(\K)$.
In particular, $\calS$ contains every matrix of the form
$$\begin{bmatrix}
0 & x & [0]_{1 \times (p-2)} \\
0 & y & [0]_{1 \times (p-2)} \\
0 & 0 & [0]_{1 \times (p-2)}
\end{bmatrix} \quad \text{with $(x,y)\in \K^2$.}$$
Thus, there is a subspace $\calT$ of $\Mat_{3,p-1}(\K)$ with codimension $2$ such that
$\calS$ is equivalent to $\calD \coprod \calT$, where $\calD$ is the space of all vectors
$\begin{bmatrix}
x \\
y \\
0
\end{bmatrix}$ with $(x,y) \in \K^2$. By Theorem \ref{maintheolin}, every range-compatible linear map on $\calT$ is local.
As this is also the case for $\calD$, we deduce that $F$ is local, contradicting our initial assumption.
This concludes the proof.
\end{proof}

\subsection{Case 1. Several rank $1$ matrices in $\calS^\bot$}\label{n=3case1}

In this section, we make the following assumption:
\begin{itemize}
\item[(A1)] The space $\calS^\bot$ contains distinct rank $1$ matrices $A$ and $B$.
\end{itemize}
We shall prove that $\calS$ has Type 1.

Combining Claims \ref{notworank1samerange} and \ref{tworank1samekernel}, we find
$$\im A \neq \im B \quad \text{and} \quad \Ker A=\Ker B.$$
By Claim \ref{intersectionkernelclaim}, no vector of $\Ker A$ is $\calS$-adapted. As the set of all non-$\calS$-adapted vectors does not span $V$,
we deduce that $\Ker A$ is exactly the set of all non-$\calS$-adapted vectors of $V$.

Let $y_3 \in V \setminus \Ker A$.
Assume that $F \modu y_3$ is local.
Choosing a basis $(y_1,y_2)$ of $\Ker A$, we know from Claim \ref{intersectionkernelclaim} that $F \modu y_1$ and $F \modu y_2$ are local,
whence Lemma \ref{3vectorslemma} would yield that $F$ is local, contradicting our assumptions. Therefore,
$F \modu y_3$ is non-local. As $y_3$ is $\calS$-adapted, it follows that $F \modu y_3$ has Type 1.
In particular, varying $y_3$ shows that there is no super-$\calS$-adapted vector.
Fixing $y_3$ once and for all, we find a matrix $C \in \calS^\bot$ with $\rk C=2$ and $\Ker C=\K y_3$.

From there, we prove that $\im A+\im B=\im C$. Let indeed
$y \in \Ker A \setminus \{0\}$. Then, $A(y+y_3)=Ay_3$, $B(y+y_3)=By_3$ and $C(y+y_3)=Cy$.
However $\dim \calS^\bot (y+y_3) \leq 2$ as there is no super-$\calS$-adapted vector.
As $Ay_3$ and $By_3$ are obviously linearly independent, we deduce that $Cy \in \Vect(Ay_3,By_3)=\im A+\im B$.
Since $V=\Ker A \oplus \Ker C$, varying $y$ yields $\im C\subset \im A+\im B$, and hence $\im C=\im A+\im B$
as the dimensions are equal on both sides.
As on the other hand $(A,B,C)$ is obviously linearly independent,
we obtain $\calS^\bot=\Vect(A,B,C)$.

Replacing $\calS$ with an equivalent space, we can assume that
$\im C=\K^2 \times \{0\}$, $\Ker A=\Ker B=\K^2 \times \{0\}$ and $\Ker C=\{0\} \times \K$.
Then, $\calS^\bot$ contains every matrix of the form
$$\begin{bmatrix}
0 & 0 & ? \\
0 & 0 & ? \\
[0]_{(p-2) \times 1} & [0]_{(p-2) \times 1} & [0]_{(p-2) \times 1}
\end{bmatrix}.$$
Moreover
$$C=\begin{bmatrix}
K & [0]_{2 \times 1} \\
[0]_{(p-2) \times 2} & [0]_{(p-2) \times 1}
\end{bmatrix}$$
for some rank $2$ matrix $K$.
Then, changing the chosen basis of $U$ once more, we can assume that $K=\begin{bmatrix}
0 & 1 \\
1 & 0
\end{bmatrix}$ on top of the previous assumptions. From there, it follows that $\calS^\bot$ is the set of all matrices of the form
$$\begin{bmatrix}
0 & a & b \\
a & 0 & c \\
[0]_{(p-2) \times 1} & [0]_{(p-2) \times 1} & [0]_{(p-2) \times 1}
\end{bmatrix} \quad \text{with $(a,b,c)\in \K^3$,}$$
and we conclude that
$$\calS= \Mats_2(\F_2) \vee \Mat_{1,p-2}(\F_2).$$
Thus, $\calS$ has Type 1, as claimed.

\subsection{Case 2. Exactly one rank $1$ matrix in $\calS^\bot$}\label{n=3case2}

In this section, we make the following extra assumption:

\begin{itemize}
\item[(A2)] There is a sole rank $1$ matrix in $\calS^\bot$, denoted by $A$.
\end{itemize}

Our goal is to prove that $\calS$ has Type 3 or 4.

We start with a lemma:

\begin{claim}
There is no super-$\calS$-adapted vector.
\end{claim}

\begin{proof}
Assume that the contrary holds. By Claim \ref{atmost1superadapted}, there is a unique super-$\calS$-adapted vector,
and we denote it by $y_3$. Then, we know from Claim \ref{specrank1claim} that $y_3 \not\in \Ker A$.
It follows from Claim \ref{adaptedoutsidekernel} that no vector of $\Ker A$ is $\calS$-adapted,
whence Claim \ref{no0trans} yields $\dim \calS^\bot y=1$ for all $y \in \Ker A$.
Thus,
$$\calW:=\bigl\{N \in \calS^\bot \mapsto Ny \mid y \in \Ker A\bigr\}$$
is a $2$-dimensional space of linear operators
of rank at most $1$. Applying the classification of spaces of linear operators with rank at most $1$,
we deduce that one of the following two situations holds:
\begin{enumerate}[(i)]
\item There is a hyperplane $H$ of $\calS^\bot$ on which all the operators of $\calW$ vanish.
\item There is a $1$-dimensional subspace $D$ of $U$ that contains the range of every operator of $\calW$.
\end{enumerate}
However, if condition (i) were satisfied then we would find some $B \in H \setminus \K A$,
and $B$ would be a rank $1$ matrix of $\calS^\bot$ that is different from $A$, contradicting assumption (A2).

Thus, condition (ii) holds, and we obtain that $Ny \in D$ for all $y \in \Ker A$ and all $N \in \calS^\bot$.
In particular, every matrix of $\calS^\bot$ vanishes at some non-zero vector of $\Ker A$.
It follows that every matrix of $\calS^\bot$ has rank at most $2$, and the kernel of a rank $2$ matrix of $\calS^\bot$
must be included in $\Ker A$. As $A$ is the sole rank $1$ matrix of $\calS^\bot$, it follows that for every $y \in V \setminus \Ker A$,
no matrix of $\calS^\bot$ annihilates $y$, whence $\dim \calS^\bot y=\dim \calS^\bot=3$ and $y$ is super-$\calS$-adapted.
This would yield four super-$\calS$-adapted vectors, contradicting Claim \ref{atmost1superadapted}.
We conclude that there is no super-$\calS$-adapted vector.
\end{proof}

As an immediate consequence of the above result and of Claim \ref{no0trans}, we obtain:

\begin{claim}
For every non-zero vector $y \in V$, either $\dim \calS^\bot y=2$ or
$\dim \calS^\bot y=1$, whether $y$ is $\calS$-adapted or not.
\end{claim}

Now, we investigate the $\calS$-adapted vectors in $\Ker A$.

\begin{claim}\label{nonadaptedinkernel}
At least one non-zero vector $y$ of $\Ker A$ is non-$\calS$-adapted.
If $y$ is the sole non-$\calS$-adapted vector in $\Ker A \setminus \{0\}$, then $\calS^\bot y=\im A$.
\end{claim}

\begin{proof}
Assume that there are distinct $\calS$-adapted vectors $y_1$ and $y_2$ in $\Ker A$.
Then, we prove that $y_1+y_2$ is non-$\calS$-adapted and that $\calS^\bot(y_1+y_2)=\im A$,
yielding all the claimed results.

We know that there is a $2$-dimensional subspace $P$ of $V$ that contains all the non-$\calS$-adapted vectors.
By Lemma \ref{coveringlemma}, we can find a vector $y_3 \in V \setminus (P \cup \Ker A)$.
Then, $y_3$ is $\calS$-adapted; as $y_1$ is $\calS$-adapted and belongs to $\Ker A$, Claim \ref{adaptedoutsidekernel}
shows that $F \modu y_3$ is non-local.
In particular, $\calS \modu y_3$ has Type 1.
Thus, we lose no generality in assuming that $(y_1,y_2,y_3)$ is the standard basis of $\K^3$
and that every matrix $M$ of $\calS$ splits up as
$M=\begin{bmatrix}
K(M) \\
[?]_{1 \times p}
\end{bmatrix}$, and $K(\calS)=\Mats_2(\K) \coprod \Mat_{2,p-2}(\K)$.

Then, by Theorem \ref{specialRCF2} we see that, by subtracting a well-chosen local map from $F \modu y_3$,
no generality is lost in assuming that
$$F: M \longmapsto \begin{bmatrix}
m_{1,1} \\
m_{2,2} \\
g(M)
\end{bmatrix} \qquad \text{for some linear form $g : \calS \rightarrow \K$.}$$
Denote by $(x_1,\dots,x_p)$ the standard basis of $U=\K^p$.
Adding $M \mapsto Mx_1$ to $F$, we find that
$F' : M \mapsto
\begin{bmatrix}
0 \\
m_{2,2}+m_{2,1} \\
g(M)+m_{3,1}
\end{bmatrix}$ is range-compatible (and still non-local).
Thus, $F' \modu y_2$ maps every operator into the line $\K \overline{y_3}$.
Applying the last statement of Claim \ref{specrank1claim} to $F'$, we obtain $F' \modu y_2=0$,
whence $g(M)=m_{3,1}$ for all $M \in \calS$. With the same line of reasoning applied
to $y_1$ instead of $y_2$, we find that $g(M)=m_{3,2}$ for all $M \in \calS$, whence
$m_{3,1}=m_{3,2}$ for all $M \in \calS$.
Therefore, $\calS^\bot$ contains the rank $1$ matrix
$$\begin{bmatrix}
0 & 0 & 1 \\
0 & 0 & 1 \\
[0]_{(p-2) \times 1} & [0]_{(p-2) \times 1} & [0]_{(p-2) \times 1}
\end{bmatrix},$$
and by assumption (A2) this matrix equals $A$. It follows that $\im A=\K(x_1+x_2)$.

Assume now that $y_1+y_2$ is also $\calS$-adapted. Then the same line of reasoning yields $g(M)=m_{3,1}+m_{3,2}$ for all
$M \in \calS$, whence $m_{3,1}=m_{3,2}=0$ for all $M \in \calS$, contradicting the fact that
$\calS^\bot$ contains a unique rank $1$ matrix.

Thus, $y_1+y_2$ is not $\calS$-adapted.
Then, from the above shape of $\calS$, it is obvious that every rank $1$ matrix $M$ of $\calS$ with image $\K(y_1+y_2)$
must satisfy $m_{1,1}=m_{2,2}=m_{1,2}=m_{2,1}$, whence $\calS^\bot(y_1+y_2)$ contains $\K(x_1+x_2)$.
As $y_1+y_2$ is not $\calS$-adapted, we conclude that $\calS^\bot(y_1+y_2)=\K(x_1+x_2)=\im A$, as claimed.
\end{proof}

\begin{claim}\label{case2numberofnonSadapted}
Exactly one non-zero vector $y$ of $V$ is non-$\calS$-adapted;
moreover, $y \in \Ker A$ and $\calS^\bot y=\im A$.
Every matrix of $\calS^\bot$ has rank at most $2$.
\end{claim}

\begin{proof}
With the notation from Claim \ref{countingclaim}, we deduce from our assumptions and from the above results that $m_1=1$, $n_1+n_2=7$, and $n_1 \geq 1$.
Claim \ref{countingclaim} then yields $3+m_2=3n_1+7-n_1$, whence $m_2=2n_1+4$.
As $m_2 \leq 6$, we have $n_1 \leq 1$ whence $n_1=1$ and $m_2=6$.
In other words, every matrix of $\calS^\bot$ has rank at most $2$ and $V \setminus \{0\}$ contains a unique non-$\calS$-adapted vector $y$.
By Claim \ref{nonadaptedinkernel}, we must have $y \in \Ker A$ and $\calS^\bot y=\im A$.
\end{proof}

Now, we consider the reduced space $\overline{\calS^\bot}$ associated with $\calS^\bot$ and we
apply the classification of matrix spaces with upper-rank $2$ (see Section \ref{primitivesection}).
Note that by Claim \ref{no0trans}, the domain of the operators of $\overline{\calS^\bot}$ is the $3$-dimensional space $V$, and hence
by Proposition \ref{nonprimitiveprop} only three possibilities can occur:
\begin{enumerate}[(1)]
\item The sum of the ranges of the operators in $\calS^\bot$ has dimension at most 2.
\item The operator space $\calS^\bot$ is represented, in well-chosen bases, by a space of matrices of the form
$$\begin{bmatrix}
? & [?]_{1 \times 2} \\
[?]_{(p-1) \times 1} & [0]_{(p-1) \times 2}
\end{bmatrix}.$$
\item The space $\overline{\calS^\bot}$ is primitive.
\end{enumerate}

Let us immediately discard option (2).
Indeed, if it held true, then we would have a whole $2$-dimensional subspace $P$ of $V$ in which every non-zero vector is
non-$\calS$-adapted, contradicting Claim \ref{case2numberofnonSadapted}.

Thus, only two possibilities remain. We shall examine them separately in the remainder of this section.

\begin{claim}
Assume that $\overline{\calS^\bot}$ is primitive. Then, $\calS$ has Type 4.
\end{claim}

\begin{proof}
As there is a non-$\calS$-adapted non-zero vector $y$, we have $\dim(\overline{\calS^\bot}y)=\dim(\calS^\bot y)=1$ and hence
Proposition \ref{primitivetypeprop2} shows that $\overline{\calS^\bot}$ is equivalent to one of the spaces
$\calM_i$, $i=1, \ldots, 4$, listed there.
As $\dim(\overline{\calS^\bot}z) \leq 2$ for all $z \in V$, whereas, with the canonical basis of $\K^3$
denoted by $(e_1,e_2,e_3)$, one checks that $\calM_1(e_2+e_3)=\calM_3(e_1+e_3)=\calM_4(e_1+e_3)=\K^3$,
we deduce that $\overline{\calS^\bot}$ is equivalent to $\calM_2$.
Using the elementary operation $L_3 \leftarrow L_3+L_2$ and then $C_1 \leftrightarrow C_3$, we see that
$\calM_2$ is equivalent to the space
$$\Biggl\{\begin{bmatrix}
0 & c & a \\
a+b & a+b & 0 \\
b & 0 & 0
\end{bmatrix} \mid (a,b,c)\in \K^3\Biggr\}.$$
Thus, $p \geq 3$ and $\calS^\bot$
is equivalent to the space of all matrices of the form
$$\begin{bmatrix}
0 & c & a+b \\
b & b & 0 \\
a & 0 & 0 \\
[0]_{(p-3) \times 1} & [0]_{(p-3) \times 1} & [0]_{(p-3) \times 1}
\end{bmatrix} \quad \text{with $(a,b,c)\in \K^3$.}$$
Using the results from Section \ref{inequivalentspecialtype}, we deduce that $\calS$ is equivalent to
$\calG_3 \coprod \Mat_{3,p-3}(\K)$, i.e.\ it has Type 4.
\end{proof}

\begin{claim}\label{downtotype3}
Assume that there is a $2$-dimensional subspace of $U$ which contains the image of every $N \in \calS^\bot$.
Then, $\calS$ has Type 3.
\end{claim}

\begin{proof}
Without loss of generality, we may assume that $\im N \subset \K^2 \times \{0\}$ for all $N \in \calS^\bot$.
In that reduced situation, $\calS$ splits up as $\calT \coprod \Mat_{3,p-2}(\K)$ for some $3$-dimensional subspace $\calT$
of $\Mat_{3,2}(\K)$, and $F$ splits up as $f\coprod g$, where $f$ and $g$
are range-compatible linear maps, respectively, on $\calT$ and $\Mat_{3,p-2}(\K)$.
Then, $g$ is local and hence $f$ is non-local. If we prove that $\calT$ has Type 3,
then it is obvious that $\calS$ will have Type 3 as well. Thus, in the rest of the proof we can simply assume that $p=2$.

Without further loss of generality, we may assume that $\im A=\{0\} \times \K=\K x_2$, where $(x_1,x_2)$ denotes the canonical basis of $U=\K^2$.
As $A$ is the sole rank $1$ matrix of $\calS^\bot$, we see that $M \in \calS \mapsto Mx$ has rank
$3$ for all $x \in \K^2 \setminus \K x_2$, whereas the range of $M \in \calS \mapsto Mx_2$ is $\Ker A$.
Consider the operators
$$\varphi : M \in \calS \mapsto Mx_1 \quad \text{and} \quad \psi : M \in \calS \mapsto Mx_2.$$
Note that
$$\calS=\Bigl\{\begin{bmatrix}
\varphi(M) & \psi(M)
\end{bmatrix} \mid M \in \calS\Bigr\}.$$
Then, $\varphi$ and $\varphi+\psi$ are isomorphisms, whereas $\psi$ has rank $2$ and its image is $\Ker A$.
Thus, the endomorphism
$$u:=\psi \circ \varphi^{-1}\in \calL(V)$$
has rank $2$, whereas $u-\id$ is invertible. It follows that $1$ is not an eigenvalue of $u$.

Let us now consider the sole non-$\calS$-adapted vector $y \in V \setminus \{0\}$.
We know that $\calS^\bot y=\im A=\K x_2$, whence $\calS$ contains a rank $1$ matrix $M$ with image $\K y$
and kernel $\K x_2$. In particular $\psi(M)=0$, whereas $\varphi(M)=M x_1=y$, leading to $u(y)=0$.
Thus, $\Ker u=\K y$. As $\im u=\Ker A$, it follows that $\Ker u \subset \im u$ and hence $0$ is not a semi-simple eigenvalue of $u$.
Therefore, $0$ is a multiple eigenvalue of $u$, and one concludes that $u$ is triangularizable since $\dim V=3$.
As on the other hand $1$ is not an eigenvalue of $u$, we conclude that $u$ is nilpotent.
As $\rk u=2$, we deduce that, in some basis $(e_1,e_2,e_3)$ of $V$, the endomorphism $u$ is represented by
$\begin{bmatrix}
0 & 1 & 0 \\
0 & 0 & 1 \\
0 & 0 & 0
\end{bmatrix}$.
It follows that $\calS$ is the set of all matrices of the form
$$\Bigl\{\begin{bmatrix}
a.e_1+b.e_2+c.e_3 & b.e_1+c.e_2
\end{bmatrix}\mid (a,b,c)\in \K^3 \Bigr\},$$
which is obviously equivalent to $\calV_2$. Therefore, $\calS$ has Type 3.
\end{proof}

The case when $\calS^\bot$ contains a rank $1$ matrix is now settled.

\subsection{Case 3. No rank $1$ matrix in $\calS^\bot$, no super-$\calS$-adapted vector}\label{n=3case3}

In this section, we make the following additional assumption:
\begin{itemize}
\item[(A3)] The space $\calS^\bot$ contains no rank $1$ matrix, and there is no super-$\calS$-adapted vector.
\end{itemize}

From there, our aim is to prove that $\calS$ is of Type 2 or 5.

\begin{claim}
All the non-zero matrices of $\calS^\bot$ have rank $2$, and all the vectors $y \in V \setminus \{0\}$
satisfy $\dim \calS^\bot y=2$.
\end{claim}

\begin{proof}
With the notation from Claim \ref{countingclaim}, we have $m_1=0$ and $n_1+n_2=7$.
Thus, $m_2=3n_1+(7-n_1)=2n_1+7$. As $m_2 \leq 7$, the only possibility is that $n_1=0$ and $m_2=7$,
which yields the claimed results.
\end{proof}

It follows that the reduced space $\overline{\calS^\bot}$, in which the domain of the operators is $V$,
has upper-rank $2$, dimension $3$, and there is no vector $y \in V$ such that $\dim (\overline{\calS^\bot} y)=1$.
Thus, combining Propositions \ref{nonprimitiveprop} and \ref{primitivetypeprop1}, we see that one of the following
four situations holds:
\begin{enumerate}[(1)]
\item There is a $2$-dimensional subspace $P$ of $U$ which contains the image of every matrix of $\calS^\bot$.

\item The operator space $\calS^\bot$ is represented, in well-chosen bases, by a space of matrices of the form
$$\begin{bmatrix}
? & [?]_{1 \times 2} \\
[?]_{(p-1) \times 1} & [0]_{(p-1) \times 2}
\end{bmatrix}.$$

\item The operator space $\overline{\calS^\bot}$ is represented by the matrix space $\Mata_3(\F_2)$.

\item The operator space $\overline{\calS^\bot}$ is represented by the matrix space $\calU_3(\F_2)$.
\end{enumerate}

However, option (2) can be discarded as it would yield some $y \in V \setminus \{0\}$ such that $\dim (\calS^\bot y) \leq 1$.

If option (3) holds true, then $p \geq 3$ and $\calS^\bot$ is equivalent to the space of all matrices of the form
$$\begin{bmatrix}
A \\
[0]_{3 \times (p-3)}
\end{bmatrix} \quad \text{with $A \in \Mata_3(\F_2)$,}$$
and one concludes that $\calS$ is equivalent to $\Mats_3(\F_2)\coprod \Mat_{3,p-3}(\F_2)$, i.e.\ it has Type 2.

If option (4) holds true then, as $\calH_3^\bot$ is equivalent to $\calU_3(\F_2)$
(see Section \ref{inequivalentspecialtype}) we deduce that
$\calS^\bot$ is equivalent to the space of all matrices of the form
$$\begin{bmatrix}
A \\
[0]_{3 \times (p-3)}
\end{bmatrix} \quad \text{with $A \in \calH_3^\bot$,}$$
and hence $\calS$ is equivalent to $\calH_3 \coprod \Mat_{3,p-3}(\F_2)$, i.e.\ it has Type 5.

In order to conclude under assumption (A3), we assume that outcome (1) holds and we try to find a contradiction.
As in the proof of Claim \ref{downtotype3}, no generality is lost in assuming that $p=2$.
Then, we use the same strategy as in that proof.
As $\calS^\bot$ contains no rank $1$ matrix, we see that $M \in \calS \mapsto M x$ has rank $3$ for all $x \in \K^2 \setminus \{0\}$.
Denote by $(e_1,e_2)$ the standard basis of $\K^2$, and consider the isomorphisms
$$\varphi : M \in \calS \mapsto M e_1 \quad \text{and} \quad \psi :  M \in \calS \mapsto M e_2.$$
Then, $u:=\psi \circ \varphi^{-1}$ is an automorphism of $V$. Moreover, since $M \in \calS \mapsto M (e_1-e_2) \in V$ is an isomorphism,
we also obtain that $u-\id$ is an automorphism of $V$.
It follows that $u$ has no eigenvalue in $\F_2$. As $\dim V=3$, the characteristic polynomial of $u$
must then be irreducible over $\F_2$, whence $u$ is cyclic and its characteristic polynomial $\chi_u(t)$ equals either
$t^3+t+1$ or $t^3+t^2+1$ (these are the sole polynomials of degree $3$ over $\F_2$ with no root in $\F_2$).
However, as no generality is truly lost in replacing $u$ with $u-\id$ (this means that we perform an elementary column operation
on $\calS$), we see that we may assume that $\tr u=0$, in which case $\chi_u(t)=t^3+t+1$.
Thus, in a well-chosen basis of $V$, the companion matrix
$\begin{bmatrix}
0 & 0 & 1 \\
1 & 0 & 1 \\
0 & 1 & 0
\end{bmatrix}$ represents $u$. Without loss of generality, we may assume that this basis is the standard one of $V=\K^3$.
Then, we are reduced to the situation where
$$\calS=\Biggl\{
\begin{bmatrix}
a & c \\
b & a+c \\
c & b
\end{bmatrix} \mid (a,b,c)\in \K^3\Biggr\}.$$
Considering $F \modu y_2$ and noting that $\calS \modu y_2$ has Type 1, we can subtract a local map from $F$ so
as to reduce the situation to the one where
$$F : \begin{bmatrix}
a & c \\
b & a+c \\
c & b
\end{bmatrix} \mapsto \begin{bmatrix}
\alpha (a+c) \\
\beta b+\gamma (a+c) \\
0
\end{bmatrix} \quad \text{for some $(\alpha,\beta,\gamma)\in \K^3.$}$$
From there, using the identity $s^2=s$ for all $s \in \K$, we obtain:
for all $(a,b,c)\in \K^3$,
\begin{multline*}
0=\begin{vmatrix}
a & c & \alpha (a+c) \\
b & a+c & \beta b+\gamma(a+c)\\
c & b & 0
\end{vmatrix} \\
= \gamma\,abc+(\alpha+\beta+\gamma) a b+(\alpha+\beta) bc +(\alpha+\gamma) ac+(\alpha+\gamma) c.
\end{multline*}
As we are dealing with a polynomial of degree at most one in each variable, we deduce that its coefficients are all zero,
and in particular $\gamma=0$, $\alpha+\gamma=0$ and $\alpha+\beta=0$, which yields $\alpha=\beta=\gamma=0$.
Therefore, $F=0$, contradicting the assumption that $F$ should be non-local.

This completes the proof in the case when $\calS^\bot$ contains no rank $1$ matrix and there is no super-$\calS$-adapted vector.

\subsection{Case 4. No rank $1$ matrix in $\calS^\bot$, one super-$\calS$-adapted vector}\label{n=3case4}

In this section, we make the following assumption:
\begin{itemize}
\item[(A4)] The space $\calS^\bot$ contains no rank $1$ matrix, and there is a super-$\calS$-adapted vector.
\end{itemize}

Under this new assumption, we shall prove that $F$ has Type 6 or 7.
By Claim \ref{atmost1superadapted}, there is a unique super-$\calS$-adapted vector, and we denote it by $y_0$.

We can readily describe the various possibilities for the ranks of the operators in $\calS^\bot$
and in $\widehat{\calS^\bot}$:

\begin{claim}\label{possiblerankscase4}
All the non-zero vectors of $V$ are $\calS$-adapted,
and $\calS^\bot$ contains one rank $3$ matrix and six rank $2$ matrices.
\end{claim}

\begin{proof}
With the notation from Claim \ref{countingclaim}, we have $m_1=0$ and $n_1+n_2=6$ from assumption (A4).
Thus $m_2=3n_1+(6-n_1)=2n_1+6$. As $m_2 \leq 7$, the only option is that $n_1=0$ and $m_2=6$,
which is precisely the claimed result.
\end{proof}

\begin{claim}\label{nocommonimage}
The intersection of all the spaces $\calS^\bot y$, with $y \in V \setminus \{0,y_0\}$, is zero.
\end{claim}

\begin{proof}
Assume on the contrary that some non-zero vector $x$ belongs to all the spaces $\calS^\bot y$ with $y \in V \setminus \{0,y_0\}$.
Let us consider the operator space $\calT:=\widehat{\calS^\bot} \modu x$, and the canonical projection $\pi : U \rightarrow U /\K x$.
For $y \in V$, denote by $\widehat{y}$ the operator $N \in \calS^\bot \mapsto Ny$.
For every $y \in V \setminus \{0,y_0\}$, we have $\rk(\pi \circ \widehat{y})=\rk \widehat{y}-1$, and we have
$\rk(\pi \circ \widehat{y_0}) \geq 2$,
whence $\calT$ has dimension $8$ and contains exactly six rank $1$ operators.
Using Lemma \ref{quadformlemma}, we see that this is absurd: indeed, we can find a quadratic form on $\calT$
that does not vanish at the sole operator in $\calT$ which has rank greater than $1$, and that vanishes at every rank $1$ operator.
\end{proof}

The next result is the key to the rest of our study:

\begin{claim}\label{commonvectorimage}
Let $\calP$ be a non-linear affine hyperplane of $V$ which contains $y_0$.
Then, some non-zero vector $x$ belongs to all the spaces $\calS^\bot y$ with $y \in \calP \setminus \{y_0\}$.
\end{claim}

\begin{proof}
This amounts to finding a non-zero vector which belongs to the kernel of every rank $1$ matrix of $\calS$
whose image is spanned by a vector of $\calP \setminus \{y_0\}$.
Denote by $P$ the translation vector space of $\calP$. Then, we may assume that the first two
vectors of the standard basis of $\K^3$ span $P$ and that the third one is $y_0$.
As $y_0$ is super-$\calS$-adapted, $F \modu y_0$ is local, and hence we can actually assume that $F \modu y_0=0$.
As $F$ is non-zero, it follows that there exists a non-zero linear form $\varphi$ on $\Mat_{1,p}(\K)$
such that, for all $M \in \calS$, we have
$$F : M \longmapsto \begin{bmatrix}
0 \\
0 \\
\varphi(L_3(M))
\end{bmatrix},$$
where $L_3(M)$ denotes the last row of $M$.
Changing the basis of $U$ further, we can actually assume that
$$F : M \longmapsto \begin{bmatrix}
0 \\
0 \\
m_{3,1}
\end{bmatrix}.$$
In that situation, we prove that the first vector $x_1$ of the standard basis of $U$ has the required properties.
Let $M \in \calS$ be a rank $1$ matrix whose image is $\K y$ for some $y \in \calP \setminus \{y_0\}$.
Then, $F(M) \in \K y \cap \K y_0$, whence $F(M)=0$. Therefore, $m_{3,1}=0$.
Thus, $Mx_1 \in P$, whence $Mx_1=0$ as $P \cap \K y=\{0\}$.
This concludes our proof.
\end{proof}

As every non-zero vector of $V$ is $\calS$-adapted,
no non-zero vector of $V$ belongs to the kernel of two distinct non-zero matrices of $\calS^\bot$. This yields:

\begin{claim}\label{distinctkernels}
The matrices of $\calS^\bot$ have pairwise distinct kernels.
\end{claim}

From now on, we split the discussion into two main cases,
whether the first or the second one of the following two assumptions holds:
\begin{itemize}
\item[(B1)] There are distinct matrices $A$ and $B$ in $\calS^\bot$ such that $\im A=\im B$.
\item[(B2)] The matrices of $\calS^\bot$ have pairwise distinct images.
\end{itemize}

We shall prove that $\calS$ has Type 6 or 7, whether condition (B1) or condition (B2) holds.

\begin{claim}
Assume that condition (B1) holds. Then, $\calS$ has Type 6.
\end{claim}

\begin{proof}
Choose distinct matrices $A$ and $B$ in $\calS^\bot$ such that $\im A=\im B$.
As $\calS^\bot$ contains exactly one rank $3$ matrix, exactly one rank $0$ matrix, and all the other ones have rank $2$,
we obtain that $\rk A=\rk B=2$.
Then, $\im(A+B) \subset \im A$ with $A+B \in \calS^\bot \setminus \{0\}$, whence $\rk(A+B)=2$.

This means that we have $2$-dimensional subspaces $P$ and $Q$, respectively, of $\calS^\bot$ and $U$, such that
$\im N=Q$ for all $N \in P \setminus \{0\}$.
Without loss of generality, we may assume that $Q=\K^2 \times \{0\}$.
From there, we choose a basis of $P$ and extend it into a basis $\calB$ of $\calS^\bot$.
Now, for all $y \in V$, we denote by $N(y)$ the matrix representing $M \in \calS^\bot \mapsto My$ in the basis
$\calB$ and the canonical basis of $U$. It follows from our assumptions that every $N(y)$
splits up as
$$N(y)=\begin{bmatrix}
K(y) & C_1(y) \\
[0]_{(p-2) \times 2} & C_2(y)
\end{bmatrix} \quad \text{with $K(y) \in \Mat_2(\K)$, $C_1(y) \in \K^2$ and $C_2(y) \in \K^{p-2}$.}$$
Given $y \in V \setminus \{0\}$, we know from Claim \ref{distinctkernels} that $Ay \neq 0$ or $By \neq 0$,
whence $K(y) \neq 0$. Therefore, $K(V)$ is a $3$-dimensional subspace of $\Mat_2(\K)$,
and $C_2(y)=0$ and $C_1(y)=0$ whenever $K(y)=0$. This yields linear maps
$$\varphi : K(V) \rightarrow \K^2 \quad \text{and} \quad \psi : K(V) \rightarrow \K^{p-2}$$
such that
$$\forall y \in V, \quad \varphi(K(y))=C_1(y) \quad \text{and} \quad \psi(K(y))=C_2(y).$$
Moreover, $\psi$ is non-zero because otherwise we would have $\dim \calS^\bot y_0 \leq 2$.

As $K(V)$ is a linear hyperplane of $\Mat_2(\K)$, either it is equivalent to the space of all upper-triangular $2 \times 2$ matrices over $\K$,
or it is equivalent to $\Mats_2(\K)$. The first case is ruled out because it would yield
some vector $x \in \K^2$ such that $\dim K(V)x=1$, contradicting the fact that $\calS^\bot$ contains no rank $1$ matrix.
Thus, no generality is lost in assuming that $K(V)=\Mats_2(\K)$.

For all $y \in V \setminus \{0,y_0\}$, the matrix $N(y)$ has rank $2$, whence $C_2(y)=0$ if
$K(y)$ is invertible. As $\Mats_2(\K)$ contains exactly $4$ invertible matrices, it follows that
$\Ker \psi$ contains at least three non-zero vectors, whence $\rk \psi=1$ and $\Ker \psi$
is a hyperplane of $\Mats_2(\K)$.
Denote by $H$ the linear hyperplane of $V$ consisting of the vectors $y$ for which $\psi(K(y))=0$.
Then, for all $y \in H \setminus \{0\}$, we see that $\calS^\bot y\subset \K^2 \times \{0\}$, whence
$\calS^\bot y=Q$ and $y \neq y_0$. By Claim \ref{commonvectorimage}, we deduce that there is a non-zero vector $x \in \K^p$
which belongs to the range of $N(y)$ for all $y \in V \setminus (H \cup \{y_0\})$.
If $x \in Q$, then $x \in \calS^\bot y$ for all $y \in H \setminus \{0\}$, as well as for all $y \in  V \setminus (H \cup \{y_0\})$,
contradicting Claim \ref{nocommonimage}. Therefore, $x \not\in Q$, whence no generality is lost in assuming that
$x$ is the third vector of the standard basis of $\K^p$.
In particular, it follows that the range of $\psi$ must contain $\begin{bmatrix}
1 & 0 & \cdots & 0
\end{bmatrix}^T$, whence $\im \psi=\K \times \{0\}$.

We have seen that all the non-zero matrices of $\Ker \psi$
are invertible, whence $\psi(M)=\begin{bmatrix}
1 & 0 & \cdots & 0
\end{bmatrix}^T$ for every rank $1$ matrix $M \in \Mats_2(\K)$.
As the rank $1$ matrices $\begin{bmatrix}
1 & 0 \\
0 & 0
\end{bmatrix}$, $\begin{bmatrix}
0 & 0 \\
0 & 1
\end{bmatrix}$ and $\begin{bmatrix}
1 & 1 \\
1 & 1
\end{bmatrix}$ span $\Mats_2(\K)$, we deduce that
$$\psi : \begin{bmatrix}
a & b \\
b & c
\end{bmatrix} \mapsto \begin{bmatrix}
a+b+c \\
[0]_{(p-3) \times 1}
\end{bmatrix}.$$
Next, we analyze $\varphi$.
Let $y \in V$ be with $\rk K(y)=1$. Then, $\psi(K(y))=\begin{bmatrix}
1 & 0 & \cdots & 0
\end{bmatrix}^T$ and $y \neq y_0$ since $N(y_0)$ has rank $3$.
Thus, $y \in V \setminus (H \cup \{y_0\})$, to the effect that $x \in \im N(y)$.
As $x$ is the third vector of the standard basis and $N(y)$ has rank $2$, it follows that
$$\rk \begin{bmatrix}
K(y) & C_1(y)
\end{bmatrix} \leq 1,$$
and hence $C_1(y) \in \im K(y)$. On the other hand, if $\rk K(y)=2$ then it is obvious that $C_1(y) \in \im  K(y)$.
Thus, $\varphi$ is range-compatible!
Note that we alter none of our assumptions by choosing some $(\lambda,\mu)\in \K^2$ and by performing the column operation
$C_3 \leftarrow C_3+\lambda C_1+\mu C_2$ on the matrix space $\bigl\{N(y) \mid y \in V\bigr\}$
(this simply means that we change our choice of last basis vector of
$\calS^\bot$ without modifying the first two). Thus, by Proposition \ref{symmetricF2}, we see that no generality is lost in assuming that
either $\varphi=0$ or $\varphi : \begin{bmatrix}
a & b \\
b & c
\end{bmatrix} \mapsto \begin{bmatrix}
a \\
c
\end{bmatrix}$. However, the first case cannot hold since $\calS^\bot$ contains no rank $1$ matrix.
Therefore, $\widehat{\calS^\bot}$ is represented by the space of all matrices
$$\begin{bmatrix}
a & b & a \\
b & c & c \\
0 & 0 & a+b+c \\
[0]_{(p-3) \times 1} & [0]_{(p-3) \times 1} & [0]_{(p-3) \times 1}
\end{bmatrix} \quad \text{with $(a,b,c)\in \K^3$.}$$

From there, we compute that $\calS^\bot$ is equivalent to the space of all matrices
$$\begin{bmatrix}
x+z & y & 0  \\
0 & x & y+z \\
z & z & z \\
[0]_{(p-3) \times 1} & [0]_{(p-3) \times 1} & [0]_{(p-3) \times 1}
\end{bmatrix} \quad \text{with $(x,y,z)\in \K^3$.}$$
Permuting the last two columns and using the results from Section \ref{inequivalentspecialtype},
we deduce that $\calS^\bot$ is equivalent to the space of all matrices
$$\begin{bmatrix}
A \\
[0]_{(p-3) \times 3}
\end{bmatrix}\quad \text{with $A \in \calI_3^\bot$,}$$
and hence $\calS$ is equivalent to $\calI_3 \coprod \Mat_{3,p-3}(\K)$, i.e.\ it has Type 6.
\end{proof}

From now on and until the end of the section, we assume that condition (B2) holds.
Our goal is to show that $\calS$ has Type 7.

We start by sharpening our knowledge of the situation considered in Claim \ref{commonvectorimage}:
remember that, given $y \in V$, we set
$$\widehat{y} : N \in \calS^\bot \mapsto Ny.$$

\begin{claim}\label{penultiemeclaim}
Let $\calP$ be a non-linear hyperplane of $V$ which contains $y_0$.
Denote by $x$ the (sole) non-vector which belongs to $\calS^\bot y$ for all $y \in \calP \setminus \{y_0\}$,
and by $\pi$ the canonical projection of $U$ onto $U/\K x$.
Then:
\begin{enumerate}[(a)]
\item The three operators $\pi \circ \widehat{y}$, for $y \in \calP \setminus \{y_0\}$
have rank $1$, pairwise distinct images and independent kernels.
\item If $\pi \circ \widehat{y_0}$ is non-injective, then none of the kernels of the operators
$\pi \circ \widehat{y}$, for $y \in \calP \setminus \{y_0\}$, contains the one of $\pi \circ \widehat{y_0}$.
\item If no $3$-dimensional space contains the range of every matrix of $\calS^\bot$,
then $\calS$ has Type 7.
\end{enumerate}
\end{claim}

\begin{proof}
First of all, we note that $\calP \setminus \{y_0\}$ spans $V$.

Let us write $\calP \setminus \{y_0\}=\{y_1,y_2,y_3\}$, and note that $y_0=y_1+y_2+y_3$.
Let $i \in \{1,2,3\}$. We know that $\widehat{y_i}$ has rank $2$ and image
$\calS^\bot y_i$, which contains $x$. Therefore, the range of $\pi \circ \widehat{y_i}$ is $\calS^\bot y/\K x$,
whence $\pi \circ \widehat{y_i}$ has rank $1$.

By assumption (B2), the ranges of the $\pi \circ \widehat{y_k}$ operators are pairwise distinct,
and we have just shown that their kernels
are $2$-dimensional subspaces of $\calS^\bot$. If the intersection of those kernels contained a non-zero matrix $M$,
then we would have $My_i \in \K x$ for all $i \in \{1,2,3\}$, whence $\im M  \subset \K x$
as $y_1,y_2,y_3$ span $V$. As $\calS^\bot$ contains no rank $1$ matrix, this is impossible,
whence the kernels of the $\pi \circ \widehat{y_i}$ operators
form a system of independent hyperplanes of $\calS^\bot$, and in particular statement (a) is established.

Now, we may find a basis $(A_1,A_2,A_3)$ of $\calS^\bot$
such that $\Ker(\pi \circ \widehat{y_i})=\Vect(A_j,A_k)$ for all distinct $i,j,k$ in $\{1,2,3\}$.
Set $z_1,z_2,z_3$ such that $\im(\pi \circ \widehat{y_i})=\K z_i$ for all $i \in \{1,2,3\}$.
We know that $z_1,z_2,z_3$ are pairwise distinct. Now, set
$$G:=\sum_{N \in \calS^\bot} \im N.$$
We know that $\K x \subset G$, whence $G/\K x$ is the sum of all ranges of the operators $\pi \circ \widehat{y}$ with $y \in V$.
As $(y_1,y_2,y_3)$ is a basis of $V$, it follows that $G/\K x$ is the sum of all ranges of the operators $\pi \circ \widehat{y_i}$
for $i \in \{1,2,3\}$, that is
$$G/\K x=\Vect(z_1,z_2,z_3).$$
Now, note that $y_0=y_1+y_2+y_3$, whence
$(\pi \circ \widehat{y_0})A_i=z_i$ for all $i \in \{1,2,3\}$.
Therefore, $\pi \circ \widehat{y_0}$ has rank $\rk(z_1,z_2,z_3)$.
If $\rk(z_1,z_2,z_3)=3$, then statement (b) is obvious.
Assume now that $\rk(z_1,z_2,z_3)=2$. Then, $z_1,z_2,z_3$ are pairwise distinct non-zero
vectors of this space, whence $z_1+z_2+z_3=0$. It follows that the rank $2$ operator
$\pi \circ \widehat{y_0}$ vanishes at the non-zero vector $A_1+A_2+A_3$, which belongs to none of the kernels of the
operators $\pi \circ \widehat{y_i}$ for $i \in \{1,2,3\}$. This proves statement (b).

Finally, let us assume that $\dim G=4$, that is $\rk(z_1,z_2,z_3)=3$.
Then, we see that $(x,A_1y_1,A_2y_2,A_3y_3)$ is a linearly independent $4$-tuple, which we extend into a basis of $U$.
Without loss of generality, we may assume that this basis is the standard one of $U=\K^p$ and that $(y_1,y_2,y_3)$
is the standard basis of $V=\K^3$.
Note that $A_i y_j \in \K x$ for all distinct $i$ and $j$ in $\{1,2,3\}$.
Hence, for some $(\alpha,\beta,\gamma,\delta,\lambda,\mu)\in \K^6$,
$$A_1=\begin{bmatrix}
0 & \alpha & \beta \\
1 & 0 & 0 \\
0 & 0 & 0 \\
0 & 0 & 0 \\
[0]_{(p-4) \times 1} & [0]_{(p-4) \times 1} & [0]_{(p-4) \times 1}
\end{bmatrix}, \quad
A_2=\begin{bmatrix}
\gamma & 0 & \delta \\
0 & 0 & 0 \\
0 & 1 & 0 \\
0 & 0 & 0 \\
[0]_{(p-4) \times 1} & [0]_{(p-4) \times 1} & [0]_{(p-4) \times 1}
\end{bmatrix}$$
and
$$A_3=\begin{bmatrix}
\lambda & \mu & 0 \\
0 & 0 & 0 \\
0 & 0 & 0 \\
0 & 0 & 1 \\
[0]_{(p-4) \times 1} & [0]_{(p-4) \times 1} & [0]_{(p-4) \times 1}
\end{bmatrix}.$$
Noting that $A_1+A_2+A_3$ has rank $3$, we deduce from Claim \ref{possiblerankscase4} that $A_1+A_2$ has rank $2$, whence
$\beta+\delta=0$. On the other hand, as $\calS^\bot y_3$ must have dimension $2$, we have $(\beta,\delta)\neq (0,0)$,
whence $\beta=\delta=1$. With the same line of reasoning, we find $\gamma=\lambda=1$ and $\alpha=\mu=1$.
As $\Vect(A_1,A_2,A_3)=\calS^\bot$, we deduce from the results of Section \ref{inequivalentspecialtype}
that $\calS^\bot$ is the space of all matrices of the form
$$\begin{bmatrix}
A \\
[0]_{(p-4) \times 3}
\end{bmatrix} \quad \text{with $A \in \calH_4^\bot$.}$$
Therefore, $\calS=\calH_4 \coprod \Mat_{3,p-4}(\K)$, whence $\calS$ has Type 7.
\end{proof}

To conclude the proof, we establish the following result:

\begin{claim}\label{lastclaim}
No $3$-dimensional space contains $\calS^\bot y$ for all $y \in V$.
\end{claim}

\begin{proof}
We use a \emph{reductio ad absurdum} and assume that such a $3$-dimensional space exists.
Then, we lose no generality in assuming that this space is $\K^3 \times \{0\}$,
in which case $\calS$ splits as $\calT \coprod \Mat_{3,p-3}(\K)$
for some $6$-dimensional subspace $\calT$ of $\Mat_3(\K)$, and $F$ splits as $G \coprod H$,
where $G$ and $H$ are range-compatible linear maps on $\calT$ and $\Mat_{3,p-3}(\K)$,
respectively. Then $H$ is local, whence $G$ is non-local. From there, we see that the space $\calT$
satisfies conditions (A4) and (B2). Thus, we can simply assume that $p=3$ in order to find a contradiction.

In this reduced situation we have $\calS \subset \Mat_3(\K)$.

Now, as $\widehat{y_0}$ has rank $3$, we can choose respective bases of $\calS^\bot$ and $U$ in which $\widehat{y_0}$
is represented by $I_3$. Denote by $\calM$ the $3$-dimensional subspace of $\Mat_3(\K)$ representing all the operators
$\widehat{y}$ in those bases. Then:
\begin{enumerate}[(a)]
\item $I_3$ is the sole non-singular matrix of $\calM$ and all the other non-zero matrices of $\calM$ have rank $2$.
\item There is a (unique) non-zero vector $x_0$ of $\K^3$ such that $\dim \calM x_0=3$,
while $\dim \calM x=2$ for all $x \in \K^3 \setminus \{0,x_0\}$.
\item No $2$-dimensional subspace $P$ of $\K^3$ is stabilized by all the matrices of $\calM$: this follows from
assumption (B2).
\item There is a (unique) vector $z_0$ of $\K^3 \setminus \{0\}$ which belongs to the image of all the
trace $1$ matrices of $\calM$. Indeed, the set of all trace $1$ matrices in $\calM$ is a non-linear affine hyperplane
that contains $I_3$, and hence the result is a consequence of Claim \ref{commonvectorimage} (noting that $\im(I_3)=\K^3$).
Moreover, $z_0$ is an eigenvector of no trace $1$ matrix of $\calM$ except $I_3$: this is a reformulation of
point (b) of Claim \ref{penultiemeclaim}.
\end{enumerate}
As every trace zero matrix of $\calM$ is the sum of $I_3$ and of a trace $1$ matrix of $\calM$,
point (d) actually shows that $z_0$ is an eigenvector of no matrix of $\calM \setminus \{0,I_3\}$.

Let $x \in \K^3 \setminus \{0,x_0\}$. As $\dim \calM x=2$, we find a non-zero matrix $A$ of $\calM$ such that
$Ax=0$. Then, one of the matrices $A$ or $A+I_3$ belongs to $\calM \setminus \{I_3\}$, has trace $1$ and
$x$ is an eigenvector for it. Thus, each vector of $\K^3 \setminus \{0,x_0\}$
is an eigenvector for some trace $1$ matrix of $\calM \setminus \{I_3\}$. It follows in particular that $x_0=z_0$.

Choose $A \in \calM \setminus \{I_3\}$ with trace $1$. As $A$ and $A-I_3$ are singular,
we see that $A$ is triangularizable and its spectrum is $\{0,1\}$; as $\tr A=1$, we see that $1$ is a single eigenvalue of $A$
and $0$ is a double eigenvalue. Moreover, as $\rk A=2$, the matrix $A$ is not diagonalisable, whence it is similar to
$\begin{bmatrix}
0 & 1 & 0 \\
0 & 0 & 0 \\
0 & 0 & 1
\end{bmatrix}$. Thus, $\Ker A^2$ has dimension $2$, the space
$\im A$ is the sum of the eigenspaces of $A$,
and $\im A \cap \Ker A^2=\Ker A$. As $z_0$ is not an eigenvector of $A$ and as it belongs to $\im A$, we deduce that $z_0 \not\in \Ker A^2$.

Let $z \in \Ker A^2 \setminus \Ker A$.
Then, $z \not\in \{0,z_0\}$ and hence $z$ is an eigenvector of some matrix $B$ of $\calM \setminus \{I_3\}$ with trace $1$;
then $I_3$, $A$ and $B$ are distinct vectors in the affine hyperplane of trace $1$ matrices of $\calM$,
and hence $(I_3,A,B)$ is a basis of $\calM$. As all those matrices map $z$ into
$\Ker A^2$, we conclude that $\calM z \subset \Ker A^2$.

Finally, we can find two distinct vectors $z_1$ and $z_2$ in $\Ker A^2 \setminus \Ker A$,
so that $(z_1,z_2)$ is a basis of $\Ker A^2$. We deduce that the $2$-dimensional space $\Ker A^2$
is stable under all the elements of $\calM$, contradicting point (c) above.
This contradiction concludes the proof.
\end{proof}

Combining Claim \ref{lastclaim} with point (c) of Claim \ref{penultiemeclaim},
we conclude that $\calS$ has Type 7. This completes the proof of Theorem \ref{classRCF2}.

\section{Application to the algebraic reflexivity of $2$-dimensional operator spaces}\label{dim2reflexive}

In \cite{BracicKuzma}, Bra\v ci\v c and Kuzma studied algebraic reflexivity for $2$-dimensional spaces
of linear operators between finite-dimensional spaces. They showed that, if the underlying field has at least $5$ elements,
such an operator space is algebraically reflexive except in a few very special cases.
Here, we shall combine Theorem \ref{classRCF2} with Theorem 1.2 of \cite{dSPRC1} to extend their result to all fields.
Recall that an operator space $\calT \subset \calL(U,V)$ is reduced when the intersection of the kernels
of the operators in $\calT$ is $\{0\}$ and the sum of the ranges of the operators in $\calT$ is $V$.

\begin{theo}[Classification of non-reflexive $2$-dimensional operator spaces]\label{2dimnonRtheorem}
Let $U$ and $V$ be finite-dimensional vector spaces, and $\calS$ be a $2$-dimensional reduced subspace of
$\calL(U,V)$. Set
$$\calE_2:=\Biggl\{\begin{bmatrix}
a & b & 0 \\
0 & a & b
\end{bmatrix} \mid (a,b)\in \K^2\Biggr\}$$
and
$$\calE_3:=\Biggl\{\begin{bmatrix}
a & b & 0 \\
0 & a & b \\
0 & 0 & a
\end{bmatrix} \mid (a,b)\in \K^2\Biggr\}.$$
Then, $\calS$ is algebraically reflexive unless one of the following conditions holds:
\begin{enumerate}
\item[(i)] $\dim U=\dim V=2$ and the set of rank $1$ operators of $\calS$ is included in a $1$-dimensional
linear subspace of $\calS$.
\item[(ii)] $\calS$ is represented by $\calE_2$ in some bases of $U$ and $V$, and $\# \K=2$.
\item[(iii)] $\calS$ is represented by $\calE_3$ in some bases of $U$ and $V$, and $\# \K=2$.
\item[(iv)] $\calS$ is represented by $\calE_2^T$ in some bases of $U$ and $V$, and $\# \K=2$.
\end{enumerate}
Moreover, in cases (ii), (iii) and (iv), the reflexivity defect of $\calS$ equals $1$.
\end{theo}

The proof will make use of the following lemma, which follows directly from Proposition \ref{Azoffprop} of the next section.

\begin{lemma}\label{transposelemma}
Let $\calS$ be a linear subspace of $\Mat_{n,p}(\K)$.
Then, the reflexivity defect of $\calS$ equals that of $\calS^T$.
\end{lemma}

\begin{proof}[Proof of Theorem \ref{2dimnonRtheorem}]
Assume first that case (i) holds.
Then, $\calS$ contains an isomorphism $f$. Without loss of generality, we may assume that
$U=V$ and $f=\id_U$. Choose $g \in \calS \setminus \K \id_U$.
Then our assumptions show that $g$ has at most one eigenvalue.
If $g$ has no eigenvalue, then $\Vect(f(x),g(x))=U$ for all non-zero vectors $x \in U$, whence $\calR(\calS)=\calL(U)$
has dimension $4$ and $\calS$ is non-reflexive.
If $g$ has exactly one eigenvalue, then no generality is lost in assuming that $g$ is nilpotent (and non-zero).
Then, in a well-chosen basis $\calB$ of $U$, the endomorphism $g$ is represented by $\begin{bmatrix}
0 & 1 \\
0 & 0
\end{bmatrix}$. In the basis $(g,f)$ and in $\calB$, the dual operator space
$\widehat{\calS}$ is represented by the space $\calT$ of all matrices of the form
$\begin{bmatrix}
y & x \\
0 & y
\end{bmatrix}$. It is easily checked that $\begin{bmatrix}
x & y \\
0 & x
\end{bmatrix} \mapsto \begin{bmatrix}
y \\
0
\end{bmatrix}$ is a non-local range-compatible homomorphism on $\calT$, whence $\calS$ is non-reflexive.

If case (ii) holds, then, in well-chosen bases of $\calS$ and $V$, the space $\widehat{\calS}$
is represented by $\Mats_2(\K)$, on which we know that there is, up to addition of a local map, a unique non-local range-compatible linear map.

Ditto for case (iii), where $\widehat{\calS}$ is represented by $\calV_2$ (as we lose no generality in assuming that $\K=\F_2$).

If case (iv) holds, then we note that case (ii) holds for $\calS^T$, and hence Lemma \ref{transposelemma} shows
that the reflexivity defect of $\calS$ equals $1$.

\vskip 4mm
Next, we prove that in any other case the space $\calS$ is reflexive.
To this effect, we assume that none of cases (i) to (iv) holds,
we consider the space $\widehat{\calS}$ and we show that every range-compatible linear map
on it is local.

Assume first that $\# \K>2$. As $\calS$ is reduced, we have $\dim \widehat{\calS}=\dim U$, whence
Theorem \ref{maintheolin} yields that every range-compatible linear map
on $\calS$ is local whenever $\dim U \geq 3$ (as here $\dim \calL(\calS,V)=2\dim V$).
If now $\# \K=2$, as $\dim \calS=2$ we see that $\widehat{\calS}$ is not of Type 2 nor of any of Types 4 to 7.
As cases (ii) and (iii) have been dismissed, Theorem \ref{classRCF2} yields that every range-compatible linear map
on $\widehat{\calS}$ is local whenever $\dim U \geq 3$ and $\# \K=2$.

\vskip 3mm
Thus, it remains to consider the case when $\dim U \leq 2$, with an arbitrary field.

Assume first that $\dim U=1$, and let $h \in \calR(\calS)$. Choosing a non-zero vector
$x_0 \in U$, we find $(\lambda,\mu)\in \K^2$ such that $h(x_0)=\lambda \,f(x_0)+\mu\, g(x_0)$,
whence $h=\lambda f+\mu g \in \calS$ as $h$ and $\lambda f+\mu g$ are linear and $x_0$ spans $U$.

\vskip 3mm
To complete the proof, we consider the case when $\dim U=2$.
If $\dim V=1$, then it is a classical result from duality theory that every linear subspace of $\calL(U,V)$ is reflexive.
Assume that $\dim V=2$. As case (i) has been dismissed, we can find two linearly independent rank $1$ operators $f$ and $g$
in $\calS$. As $\dim V=2$, $\dim U=2$ and $\calS$ is reduced, $f$ and $g$ must have distinct images
and distinct kernels; therefore, in $(f,g)$ and a basis of $V$ adapted to the decomposition $V=\im f\oplus \im g$,
the space $\widehat{\calS}$ is represented by the space $\calT$ of all matrices of the form
$\begin{bmatrix}
a & 0 \\
0 & b
\end{bmatrix}$ with $(a,b)\in \K^2$. Noting that $\calT$ splits as $\calT=\calT_1 \coprod \calT_2$ where $\calT_1$ and $\calT_2$
are linear subspaces of $\K^2$, we deduce from Lemma \ref{dimU=1} and the Splitting Lemma
that every range-compatible linear map on $\calT$ is local.

\newpage
It remains to consider the case when $\dim U=2$ and $\dim V>2$.
We choose a basis $(f,g)$ of $\calS$ and we consider the Kronecker-Weierstrass canonical form
for the matrix pencil $f+t g$ (see Chapter XII of \cite{Gantmacher}; for a proof that the results hold
for arbitrary fields, see also \cite{Dieudonnepencil}).
Remember that the Kronecker theorem for matrix pencils states that, given finite-dimensional vector spaces $E$ and $F$
and linear maps $u : E \rightarrow F$ and $v : E \rightarrow F$, there are bases $\bfB$ and $\bfC$, respectively, of $E$ and $F$
such that $M_{\bfB,\bfC}(u)=A_1 \oplus \cdots \oplus A_N$ and $M_{\bfB,\bfC}(v)=B_1 \oplus \cdots \oplus B_N$, where each
pair of matrices $(A_i,B_i)$ is of one of the following types:
\begin{enumerate}[(i)]
\item $(P,I_n)$ for some positive integer $n$ and some $P \in \GL_n(\K)$;
\item $(I_n,J_n)$ for some positive integer $n$, where $J_n:=(\delta_{i+1,j}) \in \Mat_n(\K)$;
\item $(J_n,I_n)$ for some positive integer $n$;
\item $(L_n,L'_n)$ for some positive integer $n$, where $L_n:=(\delta_{i,j}) \in \Mat_{n,n+1}(\K)$ and
$L'_n:=(\delta_{i+1,j}) \in \Mat_{n,n+1}(\K)$;
\item $(L_n^T,(L'_n)^T)$ for some positive integer $n$.
\end{enumerate}

As $\calS$ is reduced, the canonical form of the pair $(f,g)$ contains no pair of zero blocks.
As $\dim U=2$ and $\dim V>2$, there cannot be any pair of blocks of $2 \times 2$ matrices,
nor any pair of blocks of $1 \times 2$ matrices, nor
two pairs of blocks of $1 \times 1$ matrices. Therefore, only three cases are possible:
\begin{itemize}
\item \textbf{Case I.} In well-chosen bases of $U$ and $V$ and for some $\alpha \in \K$, the operators $f$ and $g$
are represented, respectively, by the matrices
$\begin{bmatrix}
1 & 0 \\
0 & 1 \\
0 & 0
\end{bmatrix}$ and $\begin{bmatrix}
\alpha & 0 \\
0 & 0 \\
0 & 1
\end{bmatrix}$.
Replacing $g$ with $g-\alpha f$ and changing the basis of $V$, we reduce the situation to the one where $\alpha=0$.
Then, $\widehat{\calS}$ is represented by the space $\calT_1$ of all matrices of the form
$$\begin{bmatrix}
a & 0 \\
b & 0 \\
0 & b
\end{bmatrix} \quad \text{with $(a,b)\in \K^2$.}$$
Let $F : \calT_1 \rightarrow \K^3$ be a range-compatible linear map. Working row by row,
we find scalars $\lambda,\mu,\nu$ such that
$$F : \begin{bmatrix}
a & 0 \\
b & 0 \\
0 & b
\end{bmatrix} \mapsto \begin{bmatrix}
\lambda a \\
\mu b \\
\nu b
\end{bmatrix}.$$
Subtracting the local map $M \mapsto M \times \begin{bmatrix}
\mu \\
\nu
\end{bmatrix}$, we may assume that $\mu=\nu=0$.
Then, for $b=1$ and $a=1$, we deduce that
$$0=\begin{vmatrix}
1 & 0 & \lambda \\
1 & 0 & 0 \\
0 & 1 & 0
\end{vmatrix}=\lambda,$$
whence $F$ is local.

\item \textbf{Case II.} In well-chosen bases of $U$ and $V$, the operators $f$ and $g$
are represented, respectively, by the matrices
$\begin{bmatrix}
1 & 0 \\
0 & 0 \\
0 & 1 \\
0 & 0
\end{bmatrix}$ and $\begin{bmatrix}
0 & 0 \\
1 & 0 \\
0 & 0 \\
0 & 1
\end{bmatrix}$.
Then, the space $\widehat{\calS}$ is represented by the space $\calT_2$ of all matrices of the form
$$\begin{bmatrix}
a & 0 \\
0 & a \\
b & 0 \\
0 & b
\end{bmatrix} \quad \text{with $(a,b)\in \K^2$.}$$
Let $F : \calT_2 \rightarrow \K^4$ be a range-compatible linear map.
Like in Case I, we see that no generality is lost in assuming that, for some $(\lambda,\mu)\in \K^2$,
$$F : \begin{bmatrix}
a & 0 \\
0 & a \\
b & 0 \\
0 & b
\end{bmatrix} \longmapsto \begin{bmatrix}
\lambda a \\
\mu a \\
0 \\
0
\end{bmatrix}.$$
Taking $a=b=1$, we find scalars $\alpha$ and $\beta$ such that
$$\begin{bmatrix}
\lambda \\
\mu \\
0 \\
0
\end{bmatrix}=\alpha \begin{bmatrix}
1 \\
0 \\
1 \\
0
\end{bmatrix}+\beta \begin{bmatrix}
0 \\
1 \\
0 \\
1
\end{bmatrix}=\begin{bmatrix}
\alpha \\
\beta \\
\alpha \\
\beta
\end{bmatrix},$$
whence $\alpha=\beta=0$, and finally $\lambda=\mu=0$. Thus, $F$ is local.

\item \textbf{Case III.}
In well-chosen bases of $U$ and $V$, the operators $f$ and $g$
are represented, respectively, by the matrices
$\begin{bmatrix}
1 & 0 \\
0 & 1 \\
0 & 0
\end{bmatrix}$ and $\begin{bmatrix}
0 & 0 \\
1 & 0 \\
0 & 1
\end{bmatrix}$.
In other words, $\calS$ is represented by $\calE_2^T$. As case (iv) has been dismissed, we deduce that $\# \K>2$.
Then, one checks that  $\widehat{\calS}$ is also represented by $\calE_2^T$ in well-chosen bases of
$\calS$ and $V$.
Let $F$ be a range-compatible linear map on $\calE_2^T$.
Then, there are scalars $\lambda,\mu,\nu,\gamma$ such that
$$F : \begin{bmatrix}
a & 0 \\
b & a \\
0 & b
\end{bmatrix} \longmapsto \begin{bmatrix}
\lambda a \\
\mu a+\nu b \\
\gamma b
\end{bmatrix}.$$
By subtracting the local map $M \mapsto M \times \begin{bmatrix}
\lambda \\
\gamma
\end{bmatrix}$ from $F$, we see that no generality is lost in assuming that
$\lambda=\gamma=0$. Then,
$$\forall (a,b)\in \K^2, \quad 0=\begin{vmatrix}
a & 0 & 0\\
b & a & \mu a+\nu b\\
0 & b & 0
\end{vmatrix}=-\mu a^2b-\nu a b^2.$$
As $\# \K>2$, we deduce that $\mu=\nu=0$ and hence $F=0$.
\end{itemize}

In any case, we have shown that every range-compatible linear map on $\widehat{\calS}$ is local, and hence
$\calS$ is reflexive.
This completes the proof of Theorem \ref{2dimnonRtheorem}.
\end{proof}

\section{Application to the classification of large affine spaces of matrices with rank greater than $1$}\label{affinespacesection}

\subsection{The problem}

\begin{Def}
The \textbf{lower-rank} of a non-empty subset $\calV$ of $\Mat_{n,p}(\K)$
is defined as $\min \{\rk M \mid M \in \calV\}$ and denoted by $\lrk \calV$.
\end{Def}

\begin{Not}
Let $n',p',n,p$ be positive integers with $n'\leq n$ and $p'\leq p$.
Given a subset $\calX$ of $\Mat_{n',p'}(\K)$, we denote by
$i_{n,p}(\calX)$ the set of all matrices of $\Mat_{n,p}(\K)$ of the form
$$\begin{bmatrix}
A & [?]_{n' \times (p-p')} \\
[?]_{(n-n') \times p'} & [?]_{(n-n') \times (p-p')}
\end{bmatrix} \quad \text{with $A \in \calX$.}$$
We also denote by $\widetilde{\calX}^{(n,p)}$ the set of all matrices of $\Mat_{n,p}(\K)$ of the form
$$\begin{bmatrix}
A & [0]_{n' \times (p-p')} \\
[0]_{(n-n') \times p'} & [0]_{(n-n') \times (p-p')}
\end{bmatrix} \quad \text{with $A \in \calX$.}$$
\end{Not}

Let $r \in \lcro 1,\min(n,p)\rcro$. In \cite{dSPlargerank}, we have proved that the codimension of an affine subspace $\calV$
of $\Mat_{n,p}(\K)$ with lower-rank $r$ is always greater than or equal to $\dbinom{r+1}{2}$. A basic way to obtain a large affine subspace of $\Mat_{n,p}(\K)$ with lower-rank $r$ is to start from an affine subspace $\calW$ of $\Mat_r(\K)$ which is included in $\GL_r(\K)$ and to
build the space $i_{n,p}(\calW)$: it is an easy observation that $i_{n,p}(\calW)$ has lower-rank $r$ and that its codimension in
$\Mat_{n,p}(\K)$ equals the codimension of $\calW$ in $\Mat_r(\K)$. In particular, if we start from an affine subspace
$\calW$ that is included in $\GL_r(\K)$ and has codimension $\dbinom{r+1}{2}$ - which we call a \textbf{dimension-maximal}
 affine subspace of non-singular matrices of $\Mat_r(\K)$ - then we obtain a subspace with codimension
$\dbinom{r+1}{2}$ in $\Mat_{n,p}(\K)$. In \cite{dSPlargeaffinerank}, it was established that this construction yields,
up to equivalence, all the affine subspaces of $\Mat_{n,p}(\K)$ with lower-rank $r$ and with the minimal codimension $\dbinom{r+1}{2}$
provided that $\# \K>2$. We restate these results here for the sake of clarity:

\begin{theo}[See \cite{dSPlargeaffinerank}]\label{classaffineboundedbelowmorethan2}
Let $n$ and $p$ be positive integers, and $r \in \lcro 2,\min(n,p)\rcro$. Assume that $\# \K > 2$.
Let $\calV$ be an affine subspace of $\Mat_{n,p}(\K)$ with lower-rank $r$
and with codimension $\dbinom{r+1}{2}$.
Then, $\calV$ is equivalent to $i_{n,p}(\calW)$ for some dimension-maximal affine subspace $\calW$ of non-singular matrices of $\Mat_r(\K)$.
\end{theo}

Moreover, we have proved the following (much easier) result, which examines to what extent
the equivalence class of $\calW$ is determined by that of $\calV$:

\begin{prop}\label{uniquenessprop}
Let $n$ and $p$ be positive integers. Let $\calW$ and $\calW'$ be dimension-maximal affine subspaces of
non-singular matrices of $\Mat_r(\K)$, where $r \in \lcro 1,\min(n,p)\rcro$. Then,
$i_{n,p}(\calW)$ and $i_{n,p}(\calW')$ are equivalent if and only if $\calW$ and $\calW'$ are equivalent.
\end{prop}

In this statement, note that we make no specific assumption on the field $\K$: one easily checks that the proof, given in Section 2 of \cite{dSPlargeaffinerank}, does not require that $\# \K>2$.

Dimension-maximal affine subspaces of non-singular matrices of $\Mat_r(\K)$ were entirely classified in \cite{dSPlargeaffine}
for fields with more than $2$ elements. For fields with $2$ elements, no classification is known yet
for general values of $r$: for $r=2$, it is known that, up to equivalence, there are exactly two such spaces, namely
$I_2+\K\,\begin{bmatrix}
0 & 1 \\
0 & 0
\end{bmatrix}$ and
$I_2+\K\,\begin{bmatrix}
0 & 1 \\
1 & 1
\end{bmatrix}$; for $r=3$, the classification was achieved in \cite{dSPprimitiveF2} (see Theorem 5.7)
but we suspect that a generalization to greater values of $r$ might be hopeless.
Over $\F_2$, there is an additional difficulty in classifying affine spaces with lower-rank $r$ and codimension $\dbinom{r+1}{2}$,
and that is the failure of Theorem \ref{classaffineboundedbelowmorethan2} in that situation (see the examples below)!

Our aim is to solve the case $r=2$ for fields with $2$ elements by using a connection with the theory
of non-reflexive operator spaces.

\subsection{The classification}

\begin{theo}\label{classaffineboundedbelow2}
Assume that $n \geq 2$ and $p \geq 2$.
Set $C:=\begin{bmatrix}
0 & 1 \\
1 & 1
\end{bmatrix}$ and $J:=\begin{bmatrix}
0 & 1 \\
0 & 0
\end{bmatrix}$.
We define two affine spaces as follows:
\[\calF_2:=\Biggl\{
\begin{bmatrix}
a+1 & a & c \\
d & a+1 & a
\end{bmatrix} \mid (a,c,d)\in (\F_2)^3
\Biggl\}\]
and
\[\calF_3:=\Biggl\{
\begin{bmatrix}
a & d & e \\
a+b+1 & a+b & f \\
c & a+b+1 & b
\end{bmatrix} \mid (a,b,c,d,e,f)\in (\F_2)^6
\Biggl\}.
\]
\begin{enumerate}[(i)]
\item If $n \geq 3$ and $p \geq 3$, then up to equivalence
there are exactly five affine subspaces of $\Mat_{n,p}(\F_2)$ with codimension $3$ and lower-rank $2$:
$i_{n,p}\bigl(I_2+\F_2 C\bigr)$, $i_{n,p}\bigl(I_2+\F_2 J\bigr)$,
$i_{n,p}\bigl(\calF_2)$, $i_{n,p}(\calF_2^T)$ and $i_{n,p}(\calF_3)$.

\item If $n \geq 3$ and $p=2$, then up to equivalence there are exactly three affine subspaces of
$\Mat_{n,p}(\F_2)$ with codimension $3$ and lower-rank $2$:
$i_{n,p}\bigl(I_2+\F_2 C\bigr)$, $i_{n,p}\bigl(I_2+\F_2 J\bigr)$ and
$i_{n,p}\bigl(\calF_2^T)$.

\item If $n=2$ and $p \geq 3$, then up to equivalence there are exactly three affine subspaces of
$\Mat_{n,p}(\F_2)$ with codimension $3$ and lower-rank $2$:
$i_{n,p}\bigl(I_2+\F_2 C\bigr)$, $i_{n,p}\bigl(I_2+\F_2 J\bigr)$ and
$i_{n,p}\bigl(\calF_2)$.

\item Up to equivalence, $ I_2+\F_2 C$ and $I_2+\F_2 J$ are the sole affine subspaces of $\Mat_2(\F_2)$ with codimension $3$
and lower-rank $2$.
\end{enumerate}
\end{theo}

Note that statement (iv) is already known since affine subspaces of $\Mat_2(\F_2)$ with codimension $3$
and lower-rank $2$ are simply dimension-maximal affine subspaces of non-singular matrices of $\Mat_2(\F_2)$.
It is an easy exercise to show that the spaces $i_{n,p}\bigl(\calF_2)$, $i_{n,p}(\calF_2^T)$ and $i_{n,p}(\calF_3)$
are counter-examples to Theorem \ref{classaffineboundedbelowmorethan2}.

\subsection{Proof of the classification theorem}

It is time to explain the connection between affine spaces of matrices with lower-rank greater than $1$ and
non-reflexive operator spaces.
Let us first recall the following result of Azoff \cite{Azoff}:

\begin{prop}\label{Azoffprop}
Let $V$ be a linear subspace of $\Mat_{n,p}(\K)$. Then,
$V$ is spanned by its rank $1$ matrices if and only if $V^\bot$ is reflexive.
Moreover, if we denote by $V^{(1)}$ the linear subspace of $V$ spanned by its rank $1$ matrices, then
$\dim V-\dim V^{(1)}$ equals the reflexivity defect of $V^\bot$.
\end{prop}

Now, let $\calV$ be an affine subspace of $\Mat_{n,p}(\K)$ which does not contain $0$.

We can see $\calV$ as an affine hyperplane of the vector space $\Vect(\calV)$,
and denote by $V$ its translation vector space, which is a linear hyperplane of $\Vect(\calV)$.
Then, $\calV$ has lower-rank at least $r$ if and only if the span of the matrices of $\Vect(\calV)$ with rank less than $r$
is included in $V$. Conversely, if we start from a linear subspace $W$ of $\Mat_{n,p}(\K)$ such that $W^{(1)} \subsetneq W$,
then, every affine hyperplane $\calH$ of $W$ which does not contain $0$ and whose translation vector space contains $W^{(1)}$
contains no matrix with rank $0$ or $1$, and hence $\lrk \calH \geq 2$.

Thus, with the connection outlined in Proposition \ref{Azoffprop}, we can derive Theorem \ref{classaffineboundedbelow2} from
Theorem \ref{2dimnonRtheorem}.
Assume from now on that $\K=\F_2$.
Let $\calS$ be an affine subspace of $\Mat_{n,p}(\K)$ with lower-rank $2$ and codimension $3$.
Then, $V:=\Vect(\calS)$ has codimension $2$, and hence $V^\bot$ is a $2$-dimensional non-reflexive subspace of $\Mat_{p,n}(\K)$.
Applying Theorem \ref{2dimnonRtheorem} to the reduced operator space $\overline{V^\bot}$, we deduce that \emph{one and only one}
of the following situations holds:
\begin{enumerate}[(i)]
\item $V^\bot$ is equivalent to $\widetilde{\calX}^{(p,n)}$ for some linear subspace $\calX$ of $\Mat_2(\F_2)$.
\item $V^\bot$ is equivalent to $\widetilde{\calE_2}^{(p,n)}$.
\item $V^\bot$ is equivalent to $\widetilde{(\calE_2^T)}^{(p,n)}$.
\item $V^\bot$ is equivalent to $\widetilde{\calE_3}^{(p,n)}$.
\end{enumerate}
Now, we tackle each case separately.

Assume first that case (i) holds. Then, we see that $V$ is equivalent to $i_{n,p}(\calX^\bot)$.
Without loss of generality, we may then assume that $V=i_{n,p}(\calX^\bot)$.
Noting that the span of the rank $1$ matrices of $V$ then contains every matrix of $i_{n,p}(\{0\})$,
we deduce that $\calS=i_{n,p}(\calH)$ for some affine subspace $\calH$ of $\Mat_2(\F_2)$
with codimension $3$ and that contains only matrices with rank greater than $1$. In other words, $\calH$ must be a dimension-maximal affine
subspace of non-singular matrices of $\Mat_2(\F_2)$. Then, by statement (iv) of Theorem \ref{classaffineboundedbelow2} (which we have already proved),
we deduce that $\calS$ is equivalent to one and only one of the spaces $i_{n,p}(I_2+\F_2 C)$ and $i_{n,p}(I_2+\F_2 J)$.

Assume now that case (ii) holds (so that $n \geq 3$). Then, $V$ is equivalent to $i_{n,p}(\calE_2^\bot)$.
Without loss of generality, we may assume that $V=i_{n,p}(\calE_2^\bot)$.
Note that $\calE_2^\bot=\Bigl\{
\begin{bmatrix}
a  & d \\
b & a \\
c & b
\end{bmatrix}
\mid (a,b,c,d)\in (\F_2)^4\Bigr\}$. Moreover,
$E_{3,1}$, $E_{1,2}$ and $A:=\begin{bmatrix}
1 & 1 \\
1 & 1 \\
1 & 1
\end{bmatrix}$ are linearly independent rank $1$ matrices of $\calE_2^\bot$.
One deduces that the translation vector space of $\calS$ must contain $i_{n,p}(\Vect(E_{1,3},E_{2,1},A))$,
which is a hyperplane of $V$. Thus, $\calS$ is the affine subspace of $i_{n,p}(\calE_2^\bot)$
which does not contain $0$ and whose translation vector space equals $i_{n,p}(\Vect(E_{1,3},E_{2,1},A))$, that
is $\calS=i_{n,p}\bigl(\calF_2^T)$.

If case (iii) holds, a similar line of reasoning as in case (ii) yields that $\calS$ is equivalent to $i_{n,p}(\calF_2)$
(and hence $p \geq 3$).

Assume finally that case (iv) holds (so that $n \geq 3$ and $p \geq 3$). Then, $V$ is equivalent to $i_{n,p}(\calE_3^\bot)$,
and no generality is lost in assuming that $V=i_{n,p}(\calE_3^\bot)$.
Note that
$$\calE_3^\bot=\Biggl\{
\begin{bmatrix}
a & e & f \\
c & a+b & g \\
d & c & b
\end{bmatrix}
\mid (a,b,c,d,e,f,g)\in (\F_2)^7\Biggr\}.$$
One checks that the rank $1$ matrices $E_{1,3}$, $E_{1,2}$, $E_{3,1}$, $E_{2,3}$, $B_1:=\begin{bmatrix}
1 & 1 & 0 \\
1 & 1 & 0 \\
1 & 1 & 0
\end{bmatrix}$ and $B_2:=\begin{bmatrix}
0 & 0 & 0 \\
1 & 1 & 1 \\
1 & 1 & 1
\end{bmatrix}$ are linearly independent and belong to $\calE_3^\bot$. Setting $H:=\Vect(E_{1,3},E_{1,2},E_{3,1},E_{2,3},B_1,B_2)$,
we deduce that the translation vector space of $\calS$ contains $i_{n,p}(H)$, and from the equality of dimensions
we conclude that $i_{n,p}(H)$ is exactly the translation vector space of $\calS$. Thus,
$\calS$ is the affine hyperplane of $i_{n,p}(\calE_3^\bot)$ which does not contain $0$ and with
translation vector space $i_{n,p}(H)$.
Noting that $H=\Biggl\{
\begin{bmatrix}
a & e & f \\
a+b & a+b & g \\
d & a+b & b
\end{bmatrix}
\mid (a,b,d,e,f,g)\in (\F_2)^6\Biggr\}$, we deduce that $\calS=i_{n,p}(\calF_3)$.

\vskip 4mm
Conversely:
\begin{itemize}
\item We already know that $i_{n,p}(I_2+\F_2 C)$ and $i_{n,p}(I_2+\F_2 J)$ are inequivalent
affine subspaces of $\Mat_{n,p}(\F_2)$ with codimension $3$ and lower-rank $2$.
Obviously, they both fall into case (i) above.
\item Assume that $n \geq 3$ and $p \geq 3$, and set $\calS:=i_{n,p}(\calF_3)$, which is an affine subspace of $\Mat_{n,p}(\F_2)$
that does not contain $0$ and has codimension $3$. Then, we see that $\Vect(\calS)=i_{n,p}(\calE_3^\bot)$
and hence $\Vect(\calS)^\bot=\widetilde{\calE_3}^{(p,n)}$, which is non-reflexive.
Thus, the span of the rank $1$ matrices of $i_{n,p}(\calF_3)$ is included in a linear hyperplane of $i_{n,p}(\calF_3)$
and hence it equals the space $i_{n,p}\bigl(\Vect(E_{1,3},E_{1,2},E_{3,1},E_{2,3},B_1,B_2)\bigr)$.
As this space is the translation vector space of $\calS$ and as $\calS$ does not contain $0$,
we conclude that the lower-rank of $\calS$ is greater than $1$. Obviously $\lrk(\calS) \leq 2$ and hence
$\lrk(\calS)=2$. Note that $\calS$ falls into case (iv) above.
\item Assume that $p \geq 3$. Using the same line of reasoning as in the preceding point, one shows that
$i_{n,p}(\calF_2)$ is an affine subspace of $\Mat_{n,p}(\F_2)$ with codimension $3$ and lower-rank $2$,
and that it falls into case (ii) above.
\item By transposing, one deduces that if $n \geq 3$ then
$i_{n,p}(\calF_2^T)$ is an affine subspace of $\Mat_{n,p}(\F_2)$ with codimension $3$ and lower-rank $2$
and that it falls into case (iii) above.
\end{itemize}
As the equivalence class of $\Vect(\calS)^\bot$ is uniquely determined by that of $\calS$,
we conclude that the various affine spaces cited in Theorem \ref{classaffineboundedbelow2} are pairwise inequivalent.
This completes the proof of Theorem \ref{classaffineboundedbelow2}.

\begin{Rem}
With the above strategy, we can give an alternative proof of Bra\v ci\v c and Kuzma's Theorem 3.10 of \cite{BracicKuzma}.
Indeed, instead of using a classification of non-reflexive operator spaces in order to classify affine spaces of matrices with upper-rank at most $2$,
we can do the opposite! Thus, let $\K$ be a field with more than $2$ elements, and
$\calS$ be a $2$-dimensional non-reflexive subspace of $\Mat_{n,p}(\K)$. Then,
$\calS^\bot$ contains an affine hyperplane $\calH$ such that $\lrk \calH \geq 2$.
By Theorem 3 of \cite{dSPlargeaffinerank}, no generality is lost in assuming that $\calH=i_{p,n}(I_2+\K M)$,
where $M \in \Mat_2(\K)$ either equals $\begin{bmatrix}
0 & 1 \\
0 & 0
\end{bmatrix}$ or has no eigenvalue in $\K$.
Thus, $\calS^\bot=i_{p,n}(\Vect(I_2,M))$ and hence $\calS=\widetilde{V}^{(n,p)}$, where $V=\Vect(I_2,M)^\bot$.
If $M=\begin{bmatrix}
0 & 1 \\
0 & 0
\end{bmatrix}$, then one sees that $V$ is equivalent to $\Vect(I_2,M)$.
If $M$ has no eigenvalue in $\K$, then every non-zero matrix of $\Vect(I_2,M)$ is non-singular, and it easily follows that
this is also the case of every matrix of $V$ (indeed, if some matrix of $V$ had rank $1$ then we would find
a non-zero vector $X \in \K^2$ such that $N \in V \mapsto NX$ has rank at most $1$, yielding a rank $1$ matrix in $V^\bot$),
and one concludes that  $\Vect(I_2,M)^\bot$ is equivalent to $\Vect(I_2,M')$ for some $M' \in \Mat_2(\K)$ with no eigenvalue in $\K$.

Conversely, let $V$ be a linear subspace of $\Mat_2(\K)$ which equals $\Vect(I_2,M)$, where $M$ is either
$\begin{bmatrix}
0 & 1 \\
0 & 0
\end{bmatrix}$ or a matrix with no eigenvalue in $\K$. In each case, one checks that
$i_{p,n}(V^\bot)$ is not spanned by its rank $1$ matrices (in the second case, $V^\bot$
contains no rank $1$ matrix, so that the span of the rank $1$ matrices of $i_{p,n}(V^\bot)$ is included in $i_{p,n}(\{0\})$),
and hence $\widetilde{V}^{(n,p)}$ is non-reflexive.
\end{Rem}

\end{document}